\documentclass[12pt, a4paper]{amsart}
\usepackage[english]{babel}
\usepackage{amssymb,amsmath,amsthm,mathrsfs,bm}
\usepackage{color}
\usepackage[margin=1in]{geometry}
\usepackage{stmaryrd}
\usepackage{chngcntr}
\usepackage{apptools}

\newtheorem{theorem}[subsection]{Theorem}
\newtheorem{example}[subsection]{Example}

\newtheorem{definition}[subsection]{Definition}
\newtheorem{lemma}[subsection]{Lemma}
\newtheorem{remark}[subsection]{Remark}
\newtheorem{proposition}[subsection]{Proposition}
\newtheorem{corollary}[subsection]{Corollary}

\newtheorem*{claim*}{Claim}
\newtheorem*{theorem*}{Theorem}

\def\bal{\begin{aligned}}
\def\eal{\end{aligned}}
\def\be{\begin{equation}\label}
\def\ee{\end{equation}}
\def\bcs{\begin{cases}}
\def\ecs{\end{cases}}
\def\={\;=\;}
\def\+{\,+\,}
\def\-{\,-\,}

\def\Z{{\mathbb Z}}

\def\Q{{\mathbb Q}}
\def\R{{\mathbb R}}
\def\F{{\mathbb F}}

\def\T{{\mathbb T}}

\def\lb{\llbracket}
\def\rb{\rrbracket}
\def\ord{\mathrm{ord}}

\def\scrC{\mathscr{C}}
\def\scrR{\mathscr{R}}
\def\scrD{\mathscr{D}}
\def\scrV{\mathscr{V}}
\def\scrF{\mathscr{F}}
\def\Tr{{\rm Tr}}
\def\Teich{{\rm Teich}}

\def\cartier{\mathscr{C}_p}

\def\dwork{\mathscr{W}_f}

\def\span{{\rm Span}}
\def\v#1{{\bf #1}}

\def\is{\equiv}
\def\mod#1{({\rm mod}\ #1)}
\def\hat{\widehat}

\newcommand\corr[1]{{\color{blue}#1}}

\title{Dwork crystals I}
\author{Frits Beukers, Masha Vlasenko}

\address{Utrecht University }
\email{f.beukers@uu.nl}

\address{Institute of Mathematics of the Polish Academy of Sciences}
\email{m.vlasenko@impan.pl}

\thanks{
Work of Frits Beukers was supported by the Netherlands Organisation for Scientific Research (NWO), grant TOP1EW.15.313. Work of Masha Vlasenko was supported by the National Science Centre of Poland (NCN), grant UMO-2016/21/B/ST1/03084.}

\begin{document}
\maketitle

\section{Introduction}

In his study of zeta-functions of families of algebraic varieties Dwork discovered
a number of remarkable congruences for truncated solutions of Picard--Fuchs equations.
For example, let
\[
F(z)= \frac1\pi\int_1^\infty \frac{dx}{\sqrt{x(x-1)(x-z)}}=\sum_{k\ge0}\left(\frac{(1/2)_k}{k!}\right)^2z^k 
\]
be the period function  associated to the Legendre family of elliptic curves $y^2=x(x-1)(x-z)$. Here $(1/2)_k$ 
denotes the Pochhammer symbol $\Gamma(k+1/2)/\Gamma(1/2)$. Let $p$ be an
odd prime and $s$ a positive integer. Let $F_{p^s}$ be the truncation of $F$ given
by
\[
F_{p^s}(z)=\sum_{k=0}^{p^s-1}\left(\frac{(1/2)_k}{k!}\right)^2z^k. 
\]
Let $z_0$ be a $p$-adic integer and suppose $F_p(z_0)$ is a $p$-adic unit. Then
 $F_{p^s}(z_0)$ is a $p$-adic unit for all $s\ge1$ and we have
\[
F_{p^{s+1}}(z_0)/F_{p^s}(z_0)\is F_{p^s}(z_0)/F_{p^{s-1}}(z_0)\mod{p^s}.
\]
The $p$-adic unit 
$\lambda(z_0) = (-1)^{\frac{p-1}2} \underset{s \to \infty}\lim F_{p^s}(z_0)/F_{p^{s-1}}(z_0)$
is a root of the zeta function of the elliptic curve corresponding to $z_0 \mod p$. (Though it looks sligtly different, this fact is a version of~\cite[(6.29)]{dwork69}.)

In a series of papers culminating in \cite[Theorem 6.2]{Ka85} Katz developed a general theory
of such congruences and their underlying mechanism. However, his congruences 
involve formal expansion coefficients of differential forms instead of truncated power series
solutions of a differential equation. In this paper we consider a third alternative, namely coefficients
of certain powers of the polynomial defining a variety. For example,
in the case of the Legendre elliptic curve they are given by 
\[
G_{p^s}(z)=\mbox{coefficient of $(xy)^{p^s-1}$ of }
(y^2-x(x-1)(x-z))^{p^s-1}.
\]
Although different from $F_{p^s}(z)$, they both satisfy the hypergeometric differential equation
modulo $p^s$. The congruences read
\[
G_{p^{s+1}}(z_0)/G_{p^s}(z_0)\is G_{p^s}(z_0)/G_{p^{s-1}}(z_0)\mod{p^s}
\]
and the quotients converge to the  $p$-adic unit root $\lambda(z_0)$. 
In this paper we shall deal with a generalized version of the congruences of
the latter type. A number of ideas in this paper are already present in \cite{Ka85},
but in a very different language. There will also be no smoothness assumptions
on the underlying variety. 
We plan to come back to the case of truncated power series solutions
in a later paper.

Let $R$ be a ring of characteristic zero and $p$ a prime number.
Suppose that we have a $p$th power Frobenius lift on $R$ which is a ring endomorphism $\sigma:R\to R$ with the
property that $\sigma(r)\is r^p\mod{p}$ for all $r\in R$. For example, when
$R=\Z$ is the ring of integers we can take $\sigma(r)=r$ for all $r$.
When $R=\Z[t]$ is a polynomial ring we can take $\sigma(g(t))=g(t^p)$. 

Let $f(\v x)=\sum_{i=1}^Nf_i\v x^{\v a_i}$ be a Laurent polynomial
in $x_1,\ldots,x_n$ with
$f_i\in R$ for all $i$. Here we use the vector notation $\v x^{\v e}=
x_1^{e_1}x_2^{e_2}\cdots x_n^{e_n}$. Let $\Delta \subset \R^n$ be the Newton polytope of $f(\v x)$,
which is the convex hull of its support $\{ \v a_i : f_i \ne 0\}$.
Let $J$ be the set of interior lattice points in $\Delta$ and set $g=\# J$.
We assume that $g>0$.
For any integer $m\ge1$ we define the $g\times g$-matrix $\beta_m$ with entries
\[
(\beta_m)_{\v u,\v v \in J}:=\mbox{coefficient of $\v x^{m\v v-\v u}$ of }f(\v x)^{m-1}.
\]
When $m=1$ we take for $\beta_m$ the identity matrix. We call $\beta_p$ the 
\emph{Hasse--Witt matrix} of $f$. When $\beta_p$ is invertible modulo $p$ it turns
out that $\beta_{p^s}$ is invertible modulo $p$ for every $s\ge1$. 
Note that being invertible modulo $p$ implies being invertible modulo all powers of $p$.

In \cite{MV16} it is shown that if the Hasse--Witt matrix is invertible modulo $p$, then 
\[
\beta_{p^{s+1}}\sigma(\beta_{p^s})^{-1}\is \beta_{p^{s}}\sigma(\beta_{p^{s-1}})^{-1}
\quad \mod{p^s}
\]
for every $s\ge1$. One may observe that this congruence as similar to 
the last part of Theorem 6.2 in Katz's paper \cite{Ka85}.
We believe that the merit of \cite{MV16} is that the proof of the
congruence is completely elementary.

Let $\delta$ be a derivation on $R$. Again in an elementary way, it is
shown in \cite{MV16} that if $\beta_p$ is invertible in $R$, then
\[
\delta(\beta_{p^{s+1}})\beta_{p^{s+1}}^{-1}\is\delta(\beta_{p^s})\beta_{p^s}^{-1}
\quad \mod{p^s}
\]
for every $s\ge1$.

These congruences imply the existence, for each Frobenius lift $\sigma$ and each derivation $\delta$ on $R$, of $p$-adic limit matrices $\Lambda_\sigma$ and
$N_\delta$ such that
\[
\Lambda_\sigma=\lim_{s\to\infty}\beta_{p^s}\sigma(\beta_{p^{s-1}})^{-1}
\mbox{ and }N_\delta=\lim_{s\to\infty}\delta(\beta_{p^s})\beta_{p^s}^{-1}.
\]
It is the goal of the present paper to give an interpretation of these
matrices in terms of operations with regular rational functions on $\T^n\setminus Z_f$,
the complement of the set of zeroes $Z_f = \{\v x :  f(\v x) = 0\}$ in the $n$-dimensional torus $\T^n$.
At the same time we provide an alternative proof of the congruences.

To be slightly more precise, we consider the $R$-module $\Omega_f$ of rational 
functions generated over $R$ by 
\[
(u_0-1)!\frac{\v x^{\v u}}{f(\v x)^{u_0}},
\]
where $u_0$ is a positive integer and $\v u\in (u_0\Delta) \cap \Z^n$.
Any derivation $\delta$ on $R$ 
can be extended naturally to $\Omega_f$ by setting $\delta(x_i)=0$ for all $i$. 

In this paper we construct the $R$-linear Cartier operator
$\cartier:\widehat \Omega_f\to \widehat \Omega_{f^\sigma}$,
where $\hat \Omega_f = \underset{\leftarrow}\lim (\Omega_f/ p^s \Omega_f)$ 
is the $p$-adic completion of $\Omega_f$ and $f^\sigma(\v x)=\sum_{i=1}^N f_i^\sigma \, \v x^{\v a_i}$
is simply $f$ with $\sigma$ applied to its coefficients. The Cartier operator commutes 
with any derivation $\delta$ of $R$.

The main results of this paper are Theorems \ref{main0} and \ref{main1}. 
Applied to the open set $\mu=\Delta^\circ$ of interior points of $\Delta$, 
they describe a free rank $g$ subquotient $Q_f=Q_f(\Delta^\circ)$ of $\Omega_f$
to which the Cartier operator descends and $\Lambda_\sigma$ is the (transposed) matrix that
corresponds to the $R$-linear map $\cartier: Q_f \to Q_{f^\sigma}$.
As a bonus of our considerations we also recover a version of Katz's theorem 
\cite[Theorem 6.2]{Ka85} as Theorem \ref{main2}.

Finally in this introduction we point out the connection with the de Rham cohomology of
the complement of $Z_f$. 
Define the modules $\Omega_f^n=\Omega_f{dx_{1}\over x_{1}}\wedge\cdots\wedge{dx_{n}\over x_{n}}$ and 
$\Omega_f^{n-1} = \oplus_{i=1}^n 
\Omega_f{dx_{1}\over x_{1}}\wedge\cdots \check{{dx_{i}\over x_{i}}}\cdots \wedge{dx_{n}\over x_{n}}$
of differential $n$- and $n-1$-forms respectively.
The above mentioned $R$-module $Q_f$ is in fact a ($p$-adic) subquotient of 
\[
\dwork:=\Omega_f^n/d(\Omega_f^{n-1}).
\]
We call the latter the {\it Dwork module}. It is known due to the work of Griffiths
and Batyrev that, when $R$ is a field and $f$ satisfies certain regularity conditions
(so called {\it$\Delta$-regularity}), then $\dwork$ is isomorphic to 
the middle de Rham cohomology $H^n_{dR}(\T^n \setminus Z_f)$ 
(see Corollary~\ref{dwork-deRham} and~\cite[Theorem 7.13]{Ba93}). In particular, it is a vector space over $R$ of finite dimension. In this paper we will not assume regularity.
We also will not assume that the Newton polytope $\Delta \subset \R^n$ is of maximal dimension.
 
\section{Regular functions and formal expansion}\label{sec:diff-forms}

Let $R$ be a characteristic zero ringdomain, $f \in R[x_1^{\pm1},\ldots,x_n^{\pm 1}]$ be 
a Laurent polynomial and $\Delta \subset \R^n$ be its Newton polytope. By $C(\Delta)$ 
we denote the subset of $\R^{n+1}$ given by 
\be{delta-cone}
C(\Delta)=\{\lambda(1,u_1,\ldots,u_n)\,|\,(u_1,\ldots,u_n)\in \Delta, \lambda\ge0\},
\ee
the positive cone spanned by the Newton polytope $\Delta$ placed in $\R^{n+1}$ in the
hyperplane $u_0=1$.

The set of integral points $C(\Delta)\cap \Z^{n+1}$ is denoted by $C(\Delta)_\Z$.
Let $C(\Delta)^+_\Z = C(\Delta)_\Z \setminus \{ \v 0 \} $ be the set of non-zero
integral points in the cone. For any $(u_0,u_1,\ldots,u_n)=\v u\in C(\Delta)_\Z$ we denote 
$\v x^{\v u}=x_1^{u_1}\cdots x_n^{u_n}$ (we simply drop the component $u_0$ here,
as there is no respective variable $x_0$). 
Consider the $R$-module $\Omega_f$ of regular rational functions generated over $R$ by
\[
\omega_{\v u}:= (u_0-1)!\frac{\v x^{\v u}}{f(\v x)^{u_0}}
\]
for all $\v u\in C(\Delta)_\Z^+$. Note that $1$ is an $R$-linear combination
of $\omega_{\v u}$ with $u_0=1$, so the constant functions are also in $\Omega_f$.

We define the module $d\Omega_f$ as the $R$-span of all derivatives $x_i\frac{\partial}
{\partial x_i} \omega$ with $\omega \in \Omega_f$ and $1 \le i \le n$.
Note that $d\Omega_f\subset \Omega_f$. The quotient $R$-module 
\[
\dwork:=\Omega_f/d\Omega_f,
\]
will be called the {\it Dwork module}. 

\begin{remark}\label{whyfactorials} 
Having the extra factor $(u_0-1)!$ in the definition of $\omega_{\v u}$ appears to 
be essential in many ways when working over rings $R$ (rather than fields). 
At the end of the introduction we mentioned that the Dwork module can be 
also written in terms of differential forms as $\dwork = \Omega^n_f / d(\Omega^{n-1}_f)$.
Factors $(u_0-1)!$ in $\omega_{\v u}$ allow the so-called Griffiths--Dwork reduction when we work in $\dwork$.
This is the procedure to reduce the pole order of a form by shifting it by exact forms.
More concretely, suppose we have a form
of the shape $k!\frac{g(\v x)}{f^{k+1}}\frac{d\v x}{\v x}$ and there exist
Laurent polynomials $g_0(\v x),g_1(\v x),\ldots,g_n(\v x)$ with support in $k\Delta$ such that
$g=g_0f+\sum_{i=1}^ng_ix_i\frac{\partial f}{\partial x_i}$. Then 
\begin{eqnarray*}
k!\frac{g(\v x)}{f^{k+1}}\frac{d\v x}{\v x}&=&
k!\frac{g_0}{f^k}\frac{d\v x}{\v x}+\sum_{i=1}^nk!\frac{g_i}{f^{k+1}}
x_i\frac{\partial f}{\partial x_i}\frac{d\v x}{\v x}\\
&=&k!\frac{g_0}{f^k}\frac{d\v x}{\v x}
+\sum_{i=1}^n(k-1)!\frac{x_i}{f^k}\frac{\partial g_i}{\partial x_i}\\
&&+\sum_{i=1}^n(-1)^id\left((k-1)!\frac{g_i}{f^k}
{dx_{1}\over x_{1}}\wedge\cdots {\overset{\vee}{dx_{i}}\over x_{i}} \cdots \wedge
{dx_{n}\over x_{n}}\right)\\
&\is&k!\frac{g_0}{f^k}\frac{d\v x}{\v x}
+\sum_{i=1}^n(k-1)!\frac{x_i}{f^k}\frac{\partial g_i}{\partial x_i} \quad 
\mod{d(\Omega_f^{n-1})}.
\end{eqnarray*}
The final form is again in $\Omega_f^n$. 
Note that factorials appear in the Laplace transform in~\cite[\S 7]{Ba93}.
\end{remark}

Rational functions can be expanded as formal Laurent series. To that end we
fix a vertex $\v b$ of $\Delta$ and 
{\it assume that the coefficient of $f$ at $\v x^\v b$ is a unit in $R$}.
Denote this coefficient by $f_{\v b}$ and expand rational functions as
\[
\frac{g(\v x)}{f(\v x)^m} \= 
\frac{ g(\v x) \v x^{-m \v b} }{ f_{\v b}^{m}(1 + \ldots)^m } \=
  \frac{ g(\v x) \v x^{-m \v b} }{ f_{\v b}^{m}}\sum_{k \ge 0} h_k(\v x) ,
\]
where $h_k(\v x)$ are Laurent polynomials supported in $k(\Delta-\v b)$ for every $k$.
There are only finitely many summands contributing to each monomial in the cone
$C(\Delta-\v b) \subseteq \R^n$.
Observe that when $g(\v x)$ is supported in $m \Delta$ the formal series in the right-hand
side is itself supported in $C(\Delta-\v b)$.
(Here we need a word of caution regarding our notation. In~\eqref{delta-cone}
the polytope $\Delta$ was placed in the hyperplane $u_0=1$ in $\R^{n+1}$,
which will be our usual convention throughout the paper. Note that, with this convention,
the difference $\Delta-\v b$ is a polytope in the hyperplane $u_0=0$ and one can view
the respective cone $C(\Delta - \v b)$ as a subset of this hyperplane $ \{ u_0 = 0\} \cong \R^n$.)
Denote the ring of formal Laurent series with support in $C(\Delta-\v b)$ and coefficients
in $R$ by 
\[
\Omega_{\rm formal} \= \{ \sum_{\v k\in C(\Delta-\v b)}a_{\v k}\v x^{\v k} \;|\; a_{\v k} \in R \}.
\]
It is indeed a ring because the cone has $\v 0$ as a vertex. The above explained procedure of formal expansion defines an embedding of $\Omega_f$ into
$\Omega_{\rm formal}$ as an $R$-submodule.  Note that we do not include the choice of $\v b$
in the notation $\Omega_{\rm formal}$.

Similarly to $d \Omega_f$, the $R$-module of \emph{formal derivatives} $d\Omega_{\rm formal}$ is defined as the $R$-span of derivatives $x_i\frac{\partial}{\partial x_i}\omega$ with $1\le i\le n$ and $\omega\in\Omega_{\rm formal}$.

\begin{lemma}\label{exactcondition}
A series $\sum_{\v k\in C(\Delta-\v b)}a_{\v k}\v x^{\v k}$ is a formal derivative if and 
only if 
\[
a_{\v k}\is0\mod{\gcd(k_1,\ldots,k_n)}\quad\mbox{for all $\v k$}.
\] 
\end{lemma}

\begin{proof}
Notice that for any monomial $\v x^{\v k}$, any $i$ and any $a\in R$ we have
\[
x_i\frac{\partial}{\partial x_i}\left(a \v x^{\v k}\right)
\= ak_i\v x^{\v k}\is0\mod{k_i}.
\]
This shows the $\Rightarrow$ part. To see the reverse implication, 
write $\gcd(k_1,\ldots,k_n) = \sum_i m_i k_i$ for some $m_i \in \Z$ and note that 
\[
\sum_i m_i x_i\frac{\partial \v x^{\v k}}{\partial x_i} = 
\gcd(k_1,\ldots,k_n) \v x^{\v k}.
\]
\end{proof}

\section{Cartier operator}\label{sec:cartier}

Let us fix a prime number $p$. We define the Cartier operator $\cartier$ on $\Omega_{\rm formal}$ by
\be{cartier-on-series-def}
\cartier\left(\sum_{\v k}a_{\v k}\v x^{\v k}\right):
=\sum_{\v k}a_{p\v k}\v x^{\v k}.
\ee
Though acting on different spaces, this operation was already used in early papers of Dwork (see~$\psi$ in~\cite[\S 2]{Dwork60}) and Reich (see $\Psi$ in~\cite[\S(b)]{Reich69}).

From now on we {\it assume that $\cap_s p^s R = \{ 0 \}$},
in which case we have a well defined $p$-adic valuation 
\[
\ord_p(r) \= \sup\{ s \in \Z_{\ge 0} : r \in p^s R\} 
\]
on $R$ which extends the usual $p$-adic valuation on $\Z \subset R$. This valuation takes finite values on all non-zero elements of $R$ and satisfies the inequalities $\ord_p(r_1 r_2) \ge \ord_p(r_1 )+ \ord_p(r_2)$ and $\ord_p(r_1+r_2) \ge \min(\ord_p(r_1),\ord_p(r_2))$. We also {\it assume that $R$ is $p$-adically complete}. In particular, $R$ is a $\Z_p$-algebra and Lemma~\ref{exactcondition} can be
reformulated as

\begin{lemma}\label{exactcondition2}
A series $h\in\Omega_{\rm formal}$ is a formal derivative if and only if
$\cartier^s(h)\is0\mod{p^s}$ for all integers $s\ge1$.
\end{lemma}

One easily shows that $\cartier\circ \theta_i=p \, \theta_i\circ\cartier$ for any $\theta_i=x_i\frac{\partial}
{\partial x_i}$. 
We thus observe that the Cartier operator preserves the submodule of formal derivatives and is
divisible by $p$ on it, i.e.
\be{cartier-on-derivatives}
\cartier: d \Omega_{\rm formal} \to p \, d \Omega_{\rm formal}.
\ee
Applying this commutation identity $s$ times
yields $\cartier^s\circ \theta_i=p^s \theta_i\circ\cartier^s$, which immediately gives one
of the implications in the last lemma.  Here is another straightforward property 
of $\cartier$:

\begin{lemma}\label{cartierhandy}
Let $g,h \in \Omega_{\rm formal}$. Then $\cartier(g(\v x^p)h)=g(\v x)\cartier(h)$.
\end{lemma}

Since Cartier operators are usually defined mod $p$ in the literature,
naming $\cartier$ a Cartier lift might be more appropriate. Nevertheless we prefer to call
it the Cartier operator.
\bigskip

We now like to restrict the Cartier operator to $\Omega_f$. We will need the $p$-adic
completion $\hat \Omega_f := \underset{\leftarrow}\lim  \Omega_f/ p^s \Omega_f$.
Fix a Frobenius lift $\sigma$ on $R$: this is a ring  endomophism $\sigma: R \to R$
such that $\sigma(r) \is r^p \mod {p}$ for every $r \in R$. Our main observation is that
\begin{proposition}\label{cartier-def} 
If $p>2$ then $\cartier(\Omega_f) \subset \hat \Omega_{f^{\sigma}}$.
\end{proposition}

\begin{proof} To see this, rewrite $1/f(\v x)^{u_0}$ as 
$f(\v x)^{p\lceil u_0/p\rceil-u_0}/f(\v x)^{p\lceil u_0/p\rceil}$.
Then note that
$f(\v x)^p=f^\sigma(\v x^p)-pG(\v x)$ for some Laurent polynomial $G$
with coefficients in $R$ and support in $p\Delta$. Then we use the $p$-adic expansion
\[
\frac{\v x^{\v u}}{f(\v x)^{u_0}}=
\frac{\v x^{\v u}f(\v x)^{p\lceil u_0/p\rceil-u_0}}
{(f^\sigma(\v x^p)-pG(\v x))^{\lceil u_0/p\rceil}}
=\sum_{r\ge0}p^r{\lceil u_0/p\rceil+r-1\choose r}
\frac{G(\v x)^r}{f^\sigma(\v x^p)^{r+\lceil u_0/p\rceil}}
\v x^{\v u}f(\v x)^{p\lceil u_0/p\rceil-u_0}.
\]
Multiply this with $(u_0-1)!$ and apply $\cartier$. Using Lemma~\ref{cartierhandy} we find that
\[
\cartier(\omega_{\v u})=\sum_{r\ge0}\frac{p^r}{r!}
\frac{(u_0-1)!}{(\lceil u_0/p\rceil-1)!}
(\lceil u_0/p\rceil+r-1)!\frac{Q_r(\v x)}{f^\sigma(\v x)^{r+\lceil u_0/p\rceil}},
\]
where the $Q_r(\v x)=\cartier(G(\v x)^r\v x^{\v u}f(\v x)^{p\lceil u_0/p\rceil-u_0})$ are
Laurent polynomials in $x_1,\ldots,x_n$ with support in
$(\lceil u_0/p\rceil+r)\Delta$ and coefficients in $R$. The last formula can be rewritten as
\begin{equation}\label{cartiermatrix}
\cartier(\omega_{\v u})=\sum_{\v v\in C(\Delta)_\Z^+}F_{\v u,\v v}
\omega^\sigma_{\v v},
\end{equation}
where 
\be{cartiermatrix-entries}
F_{\v u,\v v} = \bcs \frac{p^r}{r!} \frac{(u_0-1)!}{(\lceil u_0/p\rceil-1)!} \times \text{ coefficient of } \v x^{\v v} \text{ in } Q_r(\v x), & r:=v_0-\lceil u_0/p\rceil\ge0, \\
0 , & v_0<\lceil u_0/p\rceil. \ecs
\ee
To show that $\cartier(\omega_{\v u})\in\hat\Omega_{f^\sigma}$
it suffices to show that
$\ord_p(F_{\v u,\v v})\to\infty$ as $v_0\to\infty$. To that end we observe that
\[
\ord_p(F_{\v u,\v v})\ge r-\ord_p(r!)+\ord_p\left(\frac{(u_0-1)!}
{(\lceil u_0/p\rceil-1)!}\right).
\]
It is straightforward to see that $\frac{(u_0-1)!}{(\lceil u_0/p\rceil-1)!}$ has
order $\ge \lceil u_0/p\rceil-1$ and that $\ord_p(r!)<\frac{r}{p-1}$. This gives us
\be{cartier-entries-estimate}
\ord_p(F_{\v u,\v v})\ge r+\lceil u_0/p\rceil-1-\frac{r}{p-1}
=v_0-1-\frac{r}{p-1}\ge\frac{p-2}{p-1}(v_0-1).
\ee
The latter goes to $\infty$ with $v_0$ when $p>2$.
\end{proof} 

From now on we assume that $p>2$. By Proposition~\ref{cartier-def} we have a well-defined
$R$-linear map 
\be{cartier-f}
\cartier: \hat\Omega_f \to \hat{\Omega}_{f^\sigma}.
\ee
It is in fact given by the explicit formulas~\eqref{cartiermatrix} and~\eqref{cartiermatrix-entries},
which also show that the map~\eqref{cartier-f} is independent of the choice of vertex $\v b$ of $\Delta$
at which we are doing formal expansions.

\bigskip
We shall also be interested in regular functions supported in subsets of the cone $C(\Delta)$.
For a subset $\mu \subseteq \Delta$ let us denote by 
\[
\Omega_f(\mu) \subseteq \Omega_f
\]
the $R$-module generated by functions $\omega_{\v u}$ with $\v u\in C(\mu)_\Z^+$.
Here $C(\mu) \subseteq C(\Delta)$ is the positive cone spanned by $\mu$ placed in
$\R^{n+1}$ in the hyperplane $u_0=1$, and $C(\mu)_\Z = C(\mu) \cap \Z^{n+1}$ is the
set of integral points in this cone and $C(\mu)_\Z^+=C(\mu)_\Z \setminus \{\v 0\}$
is the set of non-zero integral points. Note that $\Omega_f(\Delta)=\Omega_f$.
The respective $p$-adic completion is denoted 
$\hat\Omega_f(\mu) = \underset{\leftarrow}\lim \; \Omega_f(\mu) / p^s \Omega_f(\mu)$.

\begin{proposition}\label{topology} Define a finite topology on $\Delta$, where the
closed sets are unions of faces of any dimension. 
Let $\mu \subset \Delta$ be an open set. Then the Cartier operator $\cartier$ maps $\hat\Omega_f(\mu)$
into $\hat\Omega_{f^\sigma}(\mu)$ and derivations of $R$ map $\Omega_f(\mu)$ to itself.
\end{proposition}
\begin{proof} Since open sets are intersections of the complements of faces, it is
enough to prove our statement for $\mu$ being such a complement. Without loss of
generality we assume that $\mu^c$ is a face of $\Delta$. In this case
$C(\mu) = \Sigma^c = C(\Delta) \setminus \Sigma$ where $\Sigma = C(\mu^c)$ is the
respective face of the cone $C(\Delta)$. The $R$-module $\Omega_f(\mu)$ is generated
by functions $\omega_\v u$ with $\v u \in \Sigma^c$. To prove our proposition we recall
that the Cartier operator~\eqref{cartier-f} is given explicitly by
formula~\eqref{cartiermatrix} and one easily sees that $\v u\in\Sigma^c$ and
$F_{\v u,\v v}\ne0$ imply $\v v\in \Sigma^c$.

Let $\delta$ be a derivation of $R$ and $\v u\in C(\mu)_\Z$. Observe that in the formula
\[
\delta (\omega_\v u) = - u_0! \frac{\v x^\v u \delta(f)(\v x)}{f(\v x)^{u_0+1}} \frac{d \v x}{\v x}
\] 
the support of $\v x^{\v u}\delta(f)$ is in $\v u+\Delta$, which again lies in $C(\mu)$ when $\mu$ 
is open in our sense.
\end{proof}

\begin{definition}\label{hassewitt}
Fix a non-empty subset $\mu \subseteq \Delta$ which is {\it open} in the
topology from Proposition~\ref{topology}.
Let $\mu_\Z=\mu \cap \Z^n$ be the set of integral points in $\mu$. We assume this set
is non-empty and let $h=\#\mu_\Z$ be the number of such points.
For any integer $m\ge1$ we define the $h\times h$-matrix $\beta_m=\beta_m(\mu)$ with entries
\be{beta-matrices}
(\beta_m)_{\v u,\v v \in \mu_\Z}:=\mbox{coefficient of $\v x^{m\v v-\v u}$ of }f(\v x)^{m-1}.
\ee
When $m=1$ we take for $\beta_1(\mu)$ the identity matrix. We call $\beta_p(\mu)$ the
\emph{Hasse--Witt matrix of $f$ relative to $\mu$}. 
\end{definition}

\begin{proposition}\label{cartiermodp}
Suppose that $p>2$ 
and $\mu \subseteq \Delta$ is a non-empty subset which is {\it open} in the topology defined
in Proposition~\ref{topology}. Then 
\[
\cartier(\hat\Omega_f(\mu))\subseteq
\span_R(\omega_{\v u}^\sigma)_{\v u\in\mu_\Z}+p\,\hat\Omega_{f^\sigma}(\mu).
\]
Moreover, for any $\v u\in\mu_\Z$ we have
\[
\cartier(\omega_{\v u}) \is \sum_{\v v \in \mu_\Z}(\beta_p)_{\v u,\v v}
\omega^\sigma_{\v v} \quad \mod{p \hat\Omega_{f^\sigma}(\mu)}.
\]
\end{proposition}

\begin{proof}
From the proof of Propositions~\ref{cartier-def} and~\ref{topology}, in particular
equations~\eqref{cartiermatrix} and~\eqref{cartiermatrix-entries},
we know an expression for $\cartier(\omega_{\v u})$ as a linear combination
$\sum_{\v v \in C(\mu)_\Z^+} F_{\v u, \v v}$ for every $\v u\in\mu_\Z$.
Moreover, it follows from~\eqref{cartier-entries-estimate} that $F_{\v u,\v v}\is0\mod{p}$
when $v_0>1$. Our first statement follows immediately.
The observation that $(\beta_p)_{\v u,\v v}\is F_{\v u,\v v}\mod{p}$
whenever $v_0=1$ proves the second statement.
\end{proof}

\section{The unit-root crystal}\label{sec:convergence}

In this section we formulate the first main result of this paper. But first we need some preparations. 

\begin{definition}\label{Uf-def}
For a non-empty subset $\mu \subseteq \Delta$ which is {\it open} in the
topology from Proposition~\ref{topology}, define $U_f(\mu)=\widehat\Omega_f(\mu) \cap d\Omega_{\rm formal}$. 
\end{definition}

We call $U_f(\mu)$ the \emph{submodule of formal derivatives}.    
Differential $n$-forms associated to elements of $U_f(\mu)$ were called 
forms that 'die on formal expansion' by
Nick Katz in \cite[p.258]{Ka85}.
It turns out that one can give a characterization of $U_f(\mu)$ which does not
make any reference to $\Omega_{\rm formal}$:
 

\begin{proposition}\label{Uf-formal-exact}
With the notations as above we have 
\begin{equation}\label{Udefinition}
U_f(\mu)=\{\omega\in\widehat\Omega_f(\mu) \;|\; \cartier^s(\omega)\is0\mod{p^s \widehat \Omega_{f^{\sigma^s}}(\mu)}
\ \mbox{for all $s\ge1$}\}.
\end{equation}
\end{proposition}

\begin{proof}
Let $\omega \in \widehat \Omega_f(\mu)$. Suppose that $\cartier^s(\omega)\is0\mod{p^s \widehat \Omega_{f^{\sigma^s}}(\mu)}$ for all $s\ge1$. Then it follows from Lemma~\ref{exactcondition2}  that $\omega \in d\Omega_{\rm formal}$.

Suppose conversely that $\omega \in \widehat\Omega_f(\mu) \cap d\Omega_{\rm formal}$. From the first part of Proposition \ref{cartiermodp} it follows that $\cartier(\omega)
=\frac{A(\v x)}{f^\sigma(\v x)}+p\omega_1$ for some $\omega_1\in\hat\Omega_{f^\sigma}(\mu)$
and $A(\v x)$ a Laurent polynomial with support in $\mu_\Z$. Since $\omega\in d\Omega_{\rm formal}$
we have that the Laurent series of $\cartier(\omega)$ is divisible by $p$. This implies
that $p$ divides $A(\v x)$. Hence $\cartier (\omega) = p \omega_2$ with 
$\omega_2=\omega_1+A(\v x)/pf^\sigma(\v x)\in\widehat\Omega_{f^\sigma}(\mu)$. 
Applying this observation recursively then yields 
$\cartier^s (\omega) \in p^s \widehat \Omega_{f^{\sigma^s}}(\mu)$, which ends our proof.
\end{proof}

It is clear from Definition~\ref{Uf-def} that the Cartier operator preserves this submodule and 
is divisible by $p$ on it, that is we have
\[
\cartier: U_f(\mu) \to p \, U_{f^\sigma}(\mu).
\]
Recall that the Cartier operator commutes with the connection operations for all derivations
$\delta$ of $R$. It is then immediate from Definition~\ref{Uf-def} that all $\delta$ preserve $U_f(\mu)$.
In other words, $U_f(\mu)$ is a differential submodule of $\widehat \Omega_f(\mu)$.   

\begin{theorem}\label{main0}
Assume that the Hasse--Witt matrix $\beta_p(\mu)$ is invertible in $R$. Then the quotient 
\[
Q_f(\mu) := \widehat \Omega_f(\mu)/ d \Omega_{\rm formal}
\] 
is a free $R$-module of rank $h= \# \mu_\Z$ with a basis given by the images of $\omega_{\v u},\v u \in \mu_\Z$. 
\end{theorem}

Strictly speaking,the quotient $ \widehat \Omega_f(\mu)/ d \Omega_{\rm formal}$ should be read as $ \widehat \Omega_f(\mu)/ U_f(\mu)$ since $U_f(\mu)=\widehat\Omega_f(\mu) \cap d\Omega_{\rm formal}$. We prefer to use the former, more suggestive, notation. 

\begin{remark}\label{invertible-crystal-remark} Recall that we work under assumptions that $\cap_s p^s R = \{ 0 \}$ and $R$ is $p$-adically complete,
in which case an element of $R$ is invertible if and only if it is invertible modulo $p$.
Indeed, if $u v = 1 + p w$ then the inverse element is given by 
$u^{-1}=v(1 + p w)^{-1} = \sum_{k \ge 0} (-p)^k v w^k$. 
With this observation, we conclude from Theorem~\ref{main0} and Proposition~\ref{cartiermodp}
that the Cartier operator on the quotients 
\[
\cartier: Q_f(\mu) \to Q_{f^\sigma}(\mu)
\]
is invertible because its matrix in the bases $\{\omega_\v u\}, \{\omega_\v u^\sigma\}$
is congruent modulo $p$ to the (transposed) Hasse--Witt matrix $\beta_p(\mu)$.
\end{remark}

Later we will give an explicit $p$-adic formula for the Cartier matrices on the quotients $Q_f(\mu)$
using matrices $\beta_{p^s}(\mu)$ for $s \ge 1$ (see Theorem~\ref{main1}).
The proof of Theorem~\ref{main0} exploits the $p$-adic contraction property of
the Cartier operator from Proposition~\ref{cartiermodp}. The main argument is
essentially contained in the following

\begin{proposition}\label{fundamental}
Let $M_0,M_1,M_2,\ldots$ be an infinite sequence of $R$-modules and
$\phi_i:M_{i-1}\to M_i$ $R$-linear maps for all $i\ge1$.
Suppose that $\cap_{s\ge1}p^sM_i=\{0\}$ for all $i$. For each $i$ let $N_i$ be a 
submodule of $M_i$ such that $\phi_i(M_{i-1})\subset N_i+pM_i$ for all $i\ge1$.
Suppose that $N_i \cap p M_i = p N_i$ 
(equivalently, $M_i/N_i$ is $p$-torsion free) and the induced maps $\phi_i:N_{i-1}/pN_{i-1}\to N_i/pN_i$
are isomorphisms for all $i\ge1$.  
Define submodules
\[
U_i=\{\omega\in M_i| \phi_{i+s}\circ\phi_{i+s-1}\circ\cdots\circ\phi_{i+1}(\omega)\is0\mod{p^sM_{i+s}}
\ \mbox{for all $s\ge1$}\} \subset M_i. 
\]

Then, for all $i$,
\begin{enumerate}
\item[(i)] $M_i=N_i+U_i$. 
\item[(ii)] $\phi_i(U_{i-1})\subset p U_i$.
\item[(iii)] $\phi_i(M_{i-1})\subset N_i+p U_i$.
\item[(iv)] $N_i\cap U_i=\{0\}$.
\end{enumerate}
\end{proposition}

\begin{proof} Note that~(ii) is an immediate consequence of the definition of $U_i$'s.
Indeed, for $\omega\in U_{i-1}$ the element $\omega_1\in M_i$ such that $\phi_i(\omega)=p\omega_1$
satisfies $\phi_{i+s}\circ\cdots\circ\phi_{i+1}(\omega_1) = 
\frac1p \phi_{i+s}\circ\cdots\circ\phi_{i}(\omega) \in p^{s}M_{i+s}$ for all $s \ge 1$.

Let us show that (iii) follows easily from (i). For any $\omega \in M_{i-1}$
write $\phi_i(\omega) = \omega_1 + p \omega_1'$ with $\omega_1 \in N_i$, $\omega_1' \in M_i$.
Using (i) we can write $\omega_1' = \nu_1 + u_1$ with $\nu_1 \in N_i$, $u_1 \in U_i$.
Thus we get $\phi_i(\omega)=\omega_1+\nu_1 + p u_1 \in N_i + p U_i$.

Proof of~(i). Clearly it is enough to do it for $i=0$.

Consider $\phi_i$ modulo $p$, which is a map from $M_{i-1}/pM_{i-1}$ to $M_i/pM_i$. By the assumption that $\phi_i(M_{i-1}) \subset N_i + p M_i$, the image of $\phi_i \mod p$ lies in $N_i / (N_i \cap p M_i)$. Since we also assume that $N_i \cap pM_i = p N_i$, the image of $\phi_i \mod p$ lies in $N_i/pN_i$. Restricting $\phi_i \mod p$ to $N_{i-1}/pN_{i-1}$ we obtain what we call \emph{the induced map} $\phi_i : N_{i-1}/pN_{i-1} \to N_i/pN_i$. It is assumed that this induced map is invertible for each $i$, and hence the composition $\psi_i:=\phi_i\circ\cdots\phi_2\circ\phi_1:N_0/p N_0\to N_i/p N_i$ is an isomorphism for all $i\ge1$. Define for each $s\ge1$
\[
U_0^{(s)}=\{\omega\in M_0| \psi_t(\omega)\in p^tM_t
\ \mbox{for all $t\le s$}\}
\]
and $U_0^{(0)}=M_0$. In particular observe that $U_0^{(s+1)}\subset U_0^{(s)}$ for all $s\ge0$. 
We first show that $U_0^{(s)}=U_0^{(s+1)}+p^sN_0$ for all $s\ge0$. Let $\omega\in U_0^{(s)}$. Then
$\omega_s:=p^{-s}\psi_s(\omega)\in M_s$. By our assumptions there
exists $\eta_s\in N_s$ such that $\phi_{s+1}(\eta_s)\is\phi_{s+1}(\omega_s)\mod{pM_{s+1}}$.
Choose $\eta_0\in N_0$ such that $\psi_s(\eta_0)\is \eta_s\mod{pM_s}$.
Then,
\begin{eqnarray*}
\psi_{s+1}(\omega-p^s\eta_0)
&\is&\phi_{s+1}(p^s\omega_s-p^s\eta_s)\mod{p^{s+1}M_{s+1}}\\
&\is&p^s(\phi_{s+1}(\omega_s)-\phi_{s+1}(\eta_s))\mod{p^{s+1}M_{s+1}}\\
&\is&0\mod{p^{s+1}M_{s+1}}.
\end{eqnarray*}
Hence $\omega-p^s\eta_0\in U_0^{(s+1)}$. 

Let $\omega\in M_0$. For $s\ge1$ we define $\omega_s\in U_0^{(s)}$ inductively via
$\omega_s=\omega_{s+1}+p^s\eta_s,\ \eta_s\in N_0$. One easily sees that
$\omega-\sum_{s\ge1}p^s\eta_s\in \cap_{s\ge1}U_0^{(s)}=U_0$. Hence we conclude that $M_0=N_0+U_0$.

We finally show that $N_0\cap U_0$ is trivial. 
Suppose, on the
contrary, that $\omega\in N_0\cap U_0$ and $\omega\ne0$. Because $\cap_{s\ge1}p^sM_0=\{ 0 \}$
there exists $s\ge0$ such that
$p^{-s}\omega\in M_0\setminus pM_0$. Since $M_0/N_0$ is $p$-torsion free this implies
that $p^{-s}\omega\in N_0\setminus pM_0$.
Since $\psi_i:N_0/pN_0\to N_i/pN_i$ is an isomorphism we
have that $\psi_i(p^{-s}\omega)\not\in pM_i$ for all $i$. In particular
for $i=s+1$ we get $\psi_{s+1}(p^{-s}\omega)\not\in pM_{s+1}$. Hence $\psi_{s+1}(\omega)
\not\in p^{s+1}M_{s+1}$. This contradicts the fact that $\omega\in U_0$. Thus we get
a contradiction and conclude that $N_0\cap U_0$ is trivial. 
\end{proof}

\begin{proof}[Proof of Theorem~\ref{main0}.]
We apply Proposition \ref{fundamental} to $M_i=\widehat\Omega_{f^{\sigma^i}}(\mu)$ and 
$\phi_i=\cartier$
for all $i\ge0$. For $N_i$ we take the $\span_R(\omega^{\sigma^i}_{\v u})_{\v u\in\mu_\Z}$. 
The property $N_i \cap p M_i = p N_i$ clearly holds. Proposition~\ref{cartiermodp} states
that $\phi_i(M_i) \subset N_i + p M_i$ and the matrix of $\phi_i: N_{i-1}/pN_{i-1} 
\to N_i / p N_i$ is given by $\beta_p^{\sigma^{i-1}}(\mu)$ modulo~$p$, which is 
invertible by the assumption in Theorem~\ref{main0}. So the assumptions of Proposition 
\ref{fundamental} are satisfied.

From Proposition~\ref{Uf-formal-exact} we find that $U_0 = U_f(\mu)$. Then application of parts (i) and (iv)
of Proposition \ref{fundamental} shows that 
\[
\hat\Omega_f(\mu)=\span_R(\omega_{\v u})_{\v u\in\mu_\Z}\oplus U_f(\mu)
\]
as $R$-modules. Hence $Q_f\cong \span_R(\omega_{\v u})_{\v u\in\mu_\Z}$.
\end{proof}

\begin{remark}\label{stronger-congruence} Parts (iii) and (iv) in Proposition~\ref{fundamental} imply that 
\[
\cartier(\widehat \Omega_f(\mu)) \subset \span_R(\omega^\sigma_{\v u})_{\v u\in\mu_\Z}\oplus p \, U_{f^\sigma}(\mu).
\]
\end{remark}

\begin{remark}\label{factorialsneeded}
Theorem \ref{main0} is not true if we would have defined
$\omega_{\v u}$ without the factorial $(u_0-1)!$.
To see this take the simplest example $f=1-x$ in one variable and $R=\Z_p$.
Theorem \ref{main0} implies that every rational function $(k-1)!\frac{x^r}{(1-x)^k}$
with $0\le r\le k$ is modulo (formal) derivatives equivalent to a function of the
form $A+\frac{B}{1-x}$. Now drop the factorial, take $k=p+1,r=0$ and suppose there
exist $A,B\in\Z_p$ such that $A+\frac{B}{1-x}=\frac{1}{(1-x)^{p+1}}+xu'$
for some rational function $u$. Apply $\cartier$ modulo $p$ on both sides. The derivative $xu'$
is mapped to $0$, $\frac{1}{1-x}$ is mapped to itself and we get
\[
A+\frac{B}{1-x}\is \cartier\left(\frac{1}{(1-x)^p(1-x)}\right)
\is \cartier\left(\frac{1}{(1-x^p)(1-x)}\right)\is \frac{1}{(1-x)^2}\mod{p}.
\]
On the right of this equality we see a rational function with a double pole at $x=1$
on the left a simple pole. This is clearly contradictory.
\end{remark}

\begin{remark} Knowledge of the explicit basis in $Q_f(\mu)$ from Theorem~\ref{main0}
implies that this $R$-module is in fact a quotient of $\Omega_f(\mu)$. However writing 
it as a quotient of the completion $\widehat \Omega_f(\mu)$ yields the Cartier operator on $Q_f(\mu)$.
Note also that $Q_f(\mu)$ is a subquotient of the Dwork module $\dwork$ because derivatives
are contained in $U_f(\mu)$.
\end{remark}

We would like to point out that $R$-modules $\widehat\Omega_f(\mu)$, $U_f(\mu)$,
completed Dwork modules $\widehat\Omega_f(\mu) / d \widehat\Omega_f(\mu)$ and the quotients
$Q_f(\mu)$ from Theorem~\ref{main0} together with the Cartier operator $\cartier$ are examples
of the following structure. For the scope of this paper, we give the following

\begin{definition}
A crystal over $R$ is a rule that assigns
\begin{itemize}
\item to a polynomial $f$ with coefficients in $R$ a differential $R$-module $M_f$, that is for 
every derivation $\delta$ of $R$ we have maps $\delta: M_f \to M_f$ satisfying 
$\delta(r m) = \delta(r) m + r \delta(m)$ for $r \in R, m \in M_f$ (\emph{connection maps});
\item to every $p$th power Frobenius lift $\sigma:R\to R$ an $R$-linear map 
$\cartier: M_f \to M_{f^\sigma}$ which commutes with the connection, that is we
have $\cartier \circ \delta = \delta \circ \cartier$ for every derivation $\delta$ of $R$.
\end{itemize} 
\end{definition}

Note that over rings $R$ which have no non-trivial derivations, e.g. $\Z_p$ and its finite extensions, 
it still makes sense to consider crystals, though the conditions related to connection are empty.
  
Following the traditional terminology, see e.g.~\cite{Ka85}, one can call $Q_f(\mu)$ the \emph{unit-root quotient}
in reflection of the fact that the Cartier operator is divisible by $p$ on $U_f(\mu)$ and invertible on the
quotient $Q_f(\mu)=\widehat\Omega_f(\mu) / U_f(\mu)$ (see Remark~\ref{invertible-crystal-remark}). 

Note also that, when the Hasse--Witt matrix is invertible, $U_f(\mu) \subset \hat \Omega_f(\mu)$ can 
be characterized as the largest subcrystal on which the Cartier operator is divisible by $p$.

\section{Periods mod $m$}\label{sec:periods}

For any exponent vector $\v v\in C(\Delta)_\Z$ we define the linear functional
$\tau_{\v v}$ on $\Omega_{\rm formal}$ by 
\[
\tau_{\v v}(\omega)=
\mbox{constant term of } \frac{f^{v_0}}{\v x^{\v v}}\omega\,.
\]

\begin{lemma}\label{zeroperiod} Let $m \ge 1$ be an integer.
For any $\omega \in d \Omega_{\rm formal}$ we have $\tau_{m\v v}(\omega)\is0\mod{m}$. 

For any $\omega \in \Omega_{\rm formal}$ and any derivation $\delta$ of $R$ we have 
$\delta(\tau_{m\v v}(\omega))\is  \tau_{m \v v}(\delta (\omega))\mod{m}$. 
\end{lemma}

\begin{proof}
Suppose that $\omega=x_i\frac{\partial u}{\partial x_i}$ for some Laurent expansion $u$.
Then 
\begin{eqnarray*}
\tau_{m\v v}(\omega)&=&\mbox{constant term of }\frac{f(\v x)^{mv_0}}{\v x^{m\v v}}
x_i\frac{\partial u}{\partial x_i}\\
&\is&\mbox{constant term of } x_i\frac{\partial}{\partial x_i}\left(\frac{f(\v x)^{mv_0}}{\v x^{m\v v}}u\right)
\mod{m}\\
&\is&0\mod{m}.
\end{eqnarray*}
For any derivation $\delta$ of $R$ and any Laurent series $\omega$ we have
\[
\delta \Bigl( \mbox{constant term of }\frac{f(\v x)^{mv_0}}{\v x^{m\v v}} 
\omega \Bigr) \is  \mbox{constant term of }\frac{f(\v x)^{mv_0}}{\v x^{m\v v}} 
\delta(\omega) 
\]
because operations $\delta$ and taking the constant term commute and derivation of an $m$th power is zero modulo $m$.  
\end{proof}

The two properties in Lemma~\ref{zeroperiod} show that functionals $\tau_{m \v v}$ restricted
modulo $m$ are what we call \emph{period maps modulo $m$}.  That is,
they are $R$-linear maps from $\Omega_{\rm formal}$ to $R/mR$ that vanish on
derivatives and commute with derivations of $R$. 

Next, we look at the behaviour of these linear functionals under the Cartier operator:   

\begin{proposition}\label{cartieraction}
Let $p$ be a prime and $\sigma: R \to R$ be a $p$th power Frobenius lift.
Denote by $\tau^\sigma_{m\v v}$ the linear functional obtained by
multiplication with $(f^{\sigma})^{mv_0}/\v x^{m\v v}$ and then taking 
the constant term. Then
$$\tau_{m\v v}\is \tau^{\sigma}_{m\v v/p}\circ\cartier\mod{p^{\ord_p(m)}}.$$
\end{proposition}

\begin{proof}
For any $\omega\in\Omega_{\rm formal}$ we have
\begin{eqnarray*}
\tau_{m\v v}(\omega)&=&
\mbox{constant term of }
\frac{f(\v x)^{mv_0}}{\v x^{m\v v}}\omega\\
&\is&\mbox{constant term of }
\frac{f^\sigma(\v x^p)^{mv_0/p}}{(\v x^p)^{m\v v/p}}
\omega\ \mod{p^{\ord_p(m)}}\\
&\is&\mbox{constant term of }\cartier\left(\frac{f^\sigma(\v x^p)^{mv_0/p}}
{(\v x^p)^{m\v v/p}}\omega\right)
\ \mod{p^{\ord_p(m)}}\\
&\is&\mbox{constant term of } \frac{f^\sigma(\v x)^{mv_0/p}}{\v x^{m\v v/p}}
\cartier(\omega)\ \mod{p^{\ord_p(m)}}.
\end{eqnarray*}
The second step uses the obvious fact that the constant term equals
the constant term of the Cartier transform.
In the last step we used a variant of Lemma~\ref{cartierhandy} in the bigger ring
$R[\v x] \otimes_R \Omega_{\rm formal}$.
\end{proof}

The period maps introduced here are useful when working in $\Omega_f$. Note that to compute $\tau_{\v v}(\omega_\v u)$ we simply take the constant coefficient of the product $\frac{f^{v_0}}{\v x^{\v v}}\omega_{\v u}$, which is a Laurent polynomial when $u_0\le v_0$. In the particular case when $u_0=v_0=1$ we observe
that $\tau_{m\v v}(\omega_{\v u})=(\beta_m)_{\v u,\v v}$ for each $m \ge 1$, where 
$\beta_m$ are the matrices defined in \eqref{beta-matrices}. 

The following theorem is our second main result.

\begin{theorem}\label{main1}Let $\mu \subseteq \Delta$ be a set open in the topology defined in
Proposition~\ref{topology}. Suppose that $R$ is $p$-adically complete and the
Hasse--Witt matrix $\beta_p(\mu)$ is invertible in $R$. Then $\beta_{p^s}(\mu)$ is invertible
for all $s\ge1$. 

Let $Q_f(\mu)$ be the unit-root crystal from Theorem~\ref{main0} and let 
$\Lambda_\sigma=(\lambda_{\v u,\v w})$ be the transposed matrix of $\cartier:Q_f(\mu)\to Q_{f^\sigma}(\mu)$
with respect to the standard bases $\{ \omega_{\v u}\},\{\omega^\sigma_{\v w}\}$.
More precisely, it is the $h\times h$-matrix with entries in $R$ such that
\begin{equation}\label{lambdamatrix0}
\cartier(\omega_{\v u})\is\sum_{\v w\in \mu_\Z}\lambda_{\v u,\v w}\omega^\sigma_{\v w}\;
\mod{U_{f^\sigma}(\mu)}
\end{equation}
for all $\v u\in\mu_\Z$. Then, for all $s\ge1$ and all $m\ge1$, it satisfies the congruences
\be{beta-lambda-congruence-with-m}
\beta_{mp^s}(\mu)\equiv \Lambda_\sigma\beta^\sigma_{mp^{s-1}}(\mu)\mod{p^s}.
\ee
In particular, when $m=1$,
\begin{equation}\label{beta-lambda-congruence-without-m}
\Lambda_\sigma \equiv \beta_{p^s}(\mu) \, \beta^\sigma_{p^{s-1}}(\mu)^{-1} \; \mod {p^s}.
\end{equation}

Similarly, for every derivation $\delta$ of $R$ Theorem~\ref{main0} implies that there exists a
unique matrix $N_{\delta}=(\nu_{\v u, \v v})_{\v u, \v v \in \mu_\Z}$ with entries in $R$ such that 
\be{numatrix}
\delta(\omega_{\v u}) \is \sum_{\v w\in \mu_\Z} \nu_{\v u,\v w} \, \omega_{\v w} \; \mod {U_f(\mu)}.
\ee
This matrix then satisfies congruences 
\[
\delta(\beta_{mp^s}(\mu))\equiv N_\delta \, \beta_{mp^s}(\mu) \; \mod {p^s}
\]
for all $m,s\ge1$. In particular, when $m=1$,
\[
N_\delta\is\delta(\beta_{p^s}(\mu)) \, \beta_{p^s}(\mu)^{-1} \; \mod {p^s}.
\]  

\end{theorem}

\begin{proof} Using (iii) in Proposition~\ref{fundamental}, the congruence~\eqref{lambdamatrix0}
 can be refined to
\be{lambdamatrix}
\cartier(\omega_{\v u})\is\sum_{\v w\in \mu_\Z}\lambda_{\v u,\v w}\omega^\sigma_{\v w}
\; \mod{p \, U_{f^\sigma}(\mu)}
\ee
(see Remark~\ref{stronger-congruence}).
We apply $\tau^\sigma_{m p^{s-1}\v v}$ with $\v v \in \mu_\Z$ to~\eqref{lambdamatrix}.
By Proposition~\ref{cartieraction} we have
\[
\tau^\sigma_{m p^{s-1}\v v}(\cartier(\omega_{\v u}))\is\tau_{m p^s\v v}(\omega_{\v u})\=
(\beta_{m p^s})_{\v u,\v v}\;\mod{p^s}
\]
in the left-hand side. Since elements of $U_{f^\sigma}(\mu)$ are formal derivatives 
(see Proposition~\ref{Uf-formal-exact}), in the right-hand side Lemma~\ref{zeroperiod} yields 
$p \, \tau^\sigma_{m p^{s-1}\v v}(U_{f^\sigma}(\mu)) \is0\mod{p^s}$. So we obtain congruence
\[
(\beta_{m p^s})_{\v u,\v v}\is\sum_{\v w\in \mu_\Z}\lambda_{\v u,\v w}
(\beta^\sigma_{m p^{s-1}})_{\v w,\v v} \; \mod{p^s}.
\]
It follows that $\beta_{m p^s}(\mu)\is\Lambda_\sigma \beta^\sigma_{m p^{s-1}}(\mu)\mod{p^s}$.
By Proposition~\ref{cartiermodp}, $\Lambda_\sigma\is\beta_p(\mu)\mod{p}$ and we find that $\beta_{p^s}(\mu)\is\beta_p(\mu)
\beta^\sigma_{p^{s-1}}(\mu)\mod{p}$. By iteration then obtain 
\[
\beta_{p^s}(\mu)\is\beta_p(\mu)\beta^\sigma_p(\mu)\cdots\beta^{\sigma^{s-1}}_p(\mu)\mod{p}.
\]
Hence invertibility of all $\beta_{p^s}(\mu)$ modulo $p$ follows from the case $s=1$. 
After inversion of $\beta^\sigma_{p^{s-1}}(\mu)$ (it is invertible over $R$, see 
Remark~\ref{invertible-crystal-remark}) we find that
$\Lambda_\sigma \equiv \beta_{p^s}(\mu) \, \beta^\sigma_{p^{s-1}}(\mu)^{-1} \; \mod {p^s}$.

The proof of the second congruence runs similarly: 
we apply $\tau_{m p^s\v v}$ with $\v v \in \mu_\Z$ to~\eqref{numatrix}. 
Since $\tau_{m p^s\v v}(U_f(\mu))\is 0 \,\mod{p^s}$ and  $\delta$ commutes 
with $\tau_{m p^s \v v}$ modulo $p^s$ (see Lemma~\ref{zeroperiod}) we obtain
\[
\delta((\beta_{m p^s})_{\v u,\v v})=\sum_{\v w\in\mu_\Z}\nu_{\v u,\v w}(\beta_{m p^{s}})_{\v w,\v v} \;
\mod{p^s}.
\]
 Hence
we conclude that $\delta(\beta_{m p^s}(\mu))\is N_\delta \, \beta_{m p^s}(\mu)\mod{p^s}$, as desired. 
\end{proof}

\begin{remark}
In~\cite[\S 1]{MV16} the second author conjectured vaguely that the $p$-adic limits 
\be{p-adic-limits-from-MV16}
\lim_{s \to \infty} \beta_{p^s}(\Delta^\circ) \, \beta_{p^{s-1}}^\sigma(\Delta^\circ)^{-1},
\quad -\lim_{s \to \infty} \delta(\beta_{p^s}(\Delta^\circ)) \, \beta_{p^{s}}(\Delta^\circ)^{-1}
\ee
describe respectively the Frobenius operator and the Gauss--Manin connection  on \emph{the}
unit-root crystal attached to the Laurent polynomial $f(\v x)$.
However the precise meaning of the unit-root crystal in the conjecture was not specified.
Moreover, it looked challenging to define this object using as little assumptions on $f(\v x)$
as one needs for existence of the $p$-adic limiting matrices~\eqref{p-adic-limits-from-MV16}.
Theorem~\ref{main1} implies that this conjecture is true with the unit-root crystal being the
dual $Q_f^\vee = {\rm Hom}_R(Q_f(\Delta^\circ),R)$ of the crystal defined in Theorem~\ref{main0}
with the Frobenius operator $\cartier^\vee: Q_{f^\sigma}^\vee \to Q_f^\vee$.
Note that in addition to the invertibility of the Hasse--Witt matrix, which is
needed to define~\eqref{p-adic-limits-from-MV16}, we only use one extra assumption:
there is a vertex $\v b$ of $\Delta$ such that the coefficient of $f(\v x)$ at $\v b$ is a
unit in $R$. The latter is a technical assumption that was made in Section~\ref{sec:diff-forms}
for the purposes of doing formal expansion at $\v b$ with integral coefficients; it is most likely
that one could drop this condition as the Cartier operator~\eqref{cartier-f} can be defined directly
by formulas~\eqref{cartiermatrix} and~\eqref{cartiermatrix-entries}.

A different proof of the conjecture was given recently in~\cite[\S 5]{HLYY18} under certain
geometric assumptions.      
\end{remark}

\begin{example} Consider $f(x,y) = y^2 - x (x-1)(x-z) \in R[x,y]$ as a polynomial with coefficients in a ring $R$ containing $\Z[z]$, which we will specify in a moment. We would like to apply Theorem~\ref{main1} with $\mu = \Delta^\circ \subset \R^2$, the interior of the Newton polytope of $f(x,y)$. In this case $\mu_\Z=\{(1,1) \}$, $h=\# \mu_\Z = 1$ and we have
\[\bal
\beta_m(\mu) & = \text{ the coefficient of } x^{m-1} y^{m-1} \text{ in } \left( y^2 - x (x-1)(x-z)\right)^{m-1} \\
&= \bcs 0, & m \text{ even}, \\
\binom{m-1}{(m-1)/2} \sum_{k=0}^{(m-1)/2}  \binom{(m-1)/2}k^2 z^k, & m \text{ odd }.
\ecs
\eal\]
To shorten our notation, we will write $\beta_m(\mu)$ simply as $\beta_m$ throughout this Example. Now fix a prime $p > 2$. Let $R = \Z[z, \beta_{p}^{-1}]\;{\hat{}} \subset \Z_p\lb z \rb$ be the $p$-adic completion of $\Z[z, \beta_{p}^{-1}]$. This ring consists of power series $g(z) \in \Z_p\lb z \rb$ that can be approximated $p$-adically by rational functions whose denominators are powers of the Hasse--Witt polynomial $\beta_p \in \Z[z]$ in the denominator. One can check that the Frobenius lift $\sigma$ given by $(\sigma g)(z) = g(z^p)$ preserves $R$. We claim that the respective Cartier matrix~\eqref{lambdamatrix0}, which is now a $1 \times 1$-matrix, is given by
\be{Lambda-Legendre-family}
\Lambda_\sigma = (-1)^{\frac{p-1}2}\frac{F(z)}{F(z^p)},
\ee 
where 
\[
F(z) \,=\, _2F_1\left(\frac12,\frac12,1 \Big| z \right) = \sum_{k\ge0}\left(\frac{(1/2)_k}{k!}\right)^2z^k  \]
is the hypergeometric series mentioned in the Introduction. Note in particular, that this statement implies that $F(z)/F(z^p) \in \Z[z, \beta_{p}^{-1}]\;{\hat{}}$.

To prove~\eqref{Lambda-Legendre-family} we notice that $\frac{(1/2)_k}{k!} = \frac{\Gamma(k+\frac12)}{\Gamma(k+1) \, \Gamma(\frac12) } = \frac{(-1)^k \Gamma(\frac12)}{\Gamma(k+1) \, \Gamma(\frac12-k) } = (-1)^k \binom{-\frac12}{k}$ and 
\[
\binom{(p^s-1)/2}{k} \equiv \binom{-\frac12}{k} \; \mod{p^{s - \ord_p(k!)}}.
\]
The latter congruence can be checked by induction on $k$. Since $\ord_p(k!) \le \frac{k}{(p-1)}$, it follows that
\[
\sum_{k=0}^{(p^s-1)/2}  \binom{(p^s-1)/2}k^2 z^k \equiv F(z) \quad \mod{ (z^s, p^{\lfloor s\frac{p-2}{p-1} \rfloor}) \Z_p\lb z \rb}.
\] 
This congruence is much weaker than the one in~\eqref{beta-lambda-congruence-without-m}. However it is sufficient to conclude that the $p$-adic limit $\Lambda_\sigma = \lim_{s \to \infty} \beta_{p^s} / \sigma(\beta_{p^{s-1}})$ equals $F(z)/F(z^p)$ times the $p$-adic limit of the ratios $\binom{p^s-1}{(p^s-1)/2}/\binom{p^{s-1}-1}{(p^{s-1}-1)/2}$. One can check that such a ratio is congruent to $(-1)^{\frac{p-1}2}$ modulo $p^s$, which completes our proof of~\eqref{Lambda-Legendre-family}. In a similar vein, one can show that $N_\delta = (\delta{F})(z)/F(z)$ for a derivation $\delta$ of $R$.  
\end{example}

Let us mention an application of congruence~\eqref{beta-lambda-congruence-with-m} to integrality of formal group laws. Consider a $h$-tuple of formal powers series $l(z) = (l_\v u(z))_{\v u \in \mu_\Z}$ in $h$ variables $z=(z_\v v)_{\v v \in \mu_\Z}$ given by
\[
l_\v u(z) \= \sum_{m=1}^{\infty} \frac1m \, \sum_{\v w \in \mu_\Z} \beta_{m}(\mu)_{\v u,\v w} \, z_\v w^m\,.
\]   
These power series have coefficients in $R \otimes \Q$ and satisfy $l_\v u(z) \is z_\v u$ modulo terms of degree $\ge 2$.

\begin{corollary}\label{fgl-corollary} Under the assumptions of Theorem~\ref{main1}, the $h$-dimensional formal group law 
\[
G(z,z') = l^{-1}(l(z)+l(z'))
\]
has coefficients in $R$.
\end{corollary} 
\begin{proof} Since $R$ is a $\Z_p$-algebra, congruences~\eqref{beta-lambda-congruence-with-m} are equivalent to the statement that the tuple of power series $l(z) - p^{-1} \Lambda_\sigma \, l^\sigma(z^p)$ has coefficients in $R$. Integrality of $G(z,z')$ then follows from Hazewinkel's functional equation lemma~\cite[\S 10.2]{Ha78}.
\end{proof}

Formal group laws $G(z,z')$ in Corollary~\ref{fgl-corollary} include coordinalizations of some Artin--Mazur formal groups of algebraic varieties, see~\cite[Theorem 1]{St87}. In the very particular example 
$f = y^2 - x(x-1)(x-z)$ from the introduction with $\mu=\Delta^\circ$ it follows
from~\cite[p. 924]{St87} that the formal group is simply the formal law of addition on the
elliptic curve $f=0$.

\bigskip

Now we would like to explain the connection between our results and~\cite{Ka85}.
For that purpose, consider linear functionals on $\Omega_{\rm formal}$ given by 
\[
\alpha_{\v k}(\omega) \= \mbox{ coefficient of }  \v x^{\v k} \mbox{ in } \omega
\]
for $\v k \in C(\Delta \corr{-} \v b)_\Z$. Just as we had above with $\tau_{\v v}$, for any $\v k$ 
and $m \ge 1$ functional $\alpha_{m \v k}$ is a period modulo $m$. Indeed, by Lemma~\ref{exactcondition}
this functional takes values in $m R$ on formal derivatives and, since derivations of $R$ act on
$\Omega_{\rm formal}$ simply by applying them to coefficients, we clearly have 
$\alpha_{m \v k} \circ \delta = \delta \circ \alpha_{m \v k}$. These periods have an
obvious property with respect to the Cartier operator:
\be{cartier-on-expansion-coeffs}
\alpha_{m \v k} \= \alpha_{m \v k/p} \circ \cartier   
\ee
for all $m$ divisible by $p$. Combining these observations with Theorem~\ref{main0} we obtain the following version of \cite[Theorem 6.2]{Ka85}:

\begin{theorem}\label{main2}  Let $\mu \subseteq \Delta$ be a set open in the
topology defined in Proposition~\ref{topology} and $h=\# \mu_\Z$. 
For $\v k \in C(\Delta \corr{-} \v b)_\Z$ consider the column vector $\v a_{\v k} \in R^h$ with components 
\[
(\v a_{\v k})_{\v u \in\mu_\Z} \= \mbox{ coefficient of } \v x^\v k \mbox{ in the formal expansion of } \omega_\v u.
\] 
Assume that $R$ is $p$-adically complete and the Hasse--Witt matrix $\beta_p(\mu)$ is invertible
in $R$. For any Frobenius lift $\sigma$ and any derivation $\delta$ of $R$, let 
$\Lambda_\sigma$ and $N_\delta$ $\Lambda_\sigma, N_\delta \in R^{h \times h}$
be the matrices defined in~\eqref{lambdamatrix0} and~\eqref{numatrix} respectively.
(These matrices correspond to the Cartier operator and connection on the unit-root crystal
$Q_f(\mu)$ defined in Theorem~\ref{main0}.) We then have
\be{frobenius}
\v a_{p^s\v k}\is\Lambda_\sigma \; \v a^\sigma_{p^{s-1}\v k}\;\mod{p^s}
\ee
and
\be{horizontality}
\delta(\v a_{p^s\v k})\is N_\delta \; \v a_{p^s\v k}\;\mod{p^s}
\ee
for all $\v k\in C(\Delta-\v b)_\Z$.

\end{theorem}

\begin{proof} Consider the equality
\[
\cartier(\omega_{\v u})\is \sum_{\v w\in\mu_\Z}\lambda_{\v u,\v w}\omega_{\v w}^\sigma 
\mod{p U^\sigma_f(\mu)}.
\]
Expand all terms in a Laurent series with respect to the vertex $\v b $ and determine the coefficient
of $\v x^{\v kp^{s-1}}$ on both sides. For the term in $pU^\sigma_f$ we get a value $0\mod{p^s}$.
The other terms give
us
\[
a_{\v kp^s}(\omega_{\v u}) \is \sum_{\v w\in\mu_\Z}\lambda_{\v u,\v w}a_{\v kp^{s-1}}(\omega^\sigma_{\v w})
\mod{p^s},
\]
which gives us the first statement.

For the second statement we start with
\[
\delta(\omega_{\v u})\is\sum_{\v w\in\mu_\Z}\nu_{\v u,\v w}\omega_{\v w}\mod{U_f(\mu)}. 
\]
Expand as Laurent series and take the coefficient of $\v x^{\v kp^s}$ on both sides. We
get
\[
\delta(a_{\v kp^s}(\omega_{\v u}))\is \sum_{\v w\in\mu_\Z}\nu_{\v u,\v w}a_{\v kp^s}(\omega_{\v w})\mod{p^s},
\]
which proves our second statement.
\end{proof}

We end with an application of Theorems \ref{main1} and \ref{main2}.

\begin{corollary}\label{special-delta}
Suppose that $\mu$ is an open set that consists of one vertex point $\v v\in\Delta$. Let $f_{\v v}$ be
the coefficient of $\v x^{\v v}$ in $f$ and suppose it is a unit in $R$. 
Then we have the equality
$\cartier(\omega_{\v v})\is\frac{f_{\v v}^\sigma}{f_{\v v}}\omega_{\v v}^\sigma
\mod{U_{f^\sigma}(\mu)}$.
\end{corollary}

\begin{proof}
This follows almost immediately from Theorem \ref{main1}. Note that $\beta_{p^s}(\mu)$ is a $1\times1$-matrix
with entry $f_{\v b}^{p^s-1}$. The matrix $\Lambda_\sigma$ has the entry
$\lim_{s\to\infty}f_{\v b}^{p^s-1}/(f_{\v b}^\sigma)^{p^{s-1}-1}=
f_{\v b}^\sigma/f_{\v b}$.
\end{proof}

Note that the situation when one vertex is an open set in the topology from Proposition~\ref{topology}
can occur if all lattice points in $\Delta$ are vertices. The complement
of all but one of these vertices gives us an open one-point set $\mu$.

The following corollary is a generalization of Theorem 5.6 in \cite{BHS18}, which
deals with congruences for coefficients of power series expansions of rational
functions.

\begin{corollary}\label{b-h-s-congruences}
Let $f(\v x)$ be a Laurent polynomial with coefficients in $\Z_p$ such that all lattice
points in its Newton polytope $\Delta \subset \R^n$ are vertices.
Suppose that all coefficients of $f(\v x)$ are $p$-adic units. Let $g(\v x)$ be a Laurent
polynomial with coefficients in $\Z_p$ and support in $\Delta$. 
Choose any vertex $\v b \in \Delta$ and consider the respective formal expansion
\[
\frac{g(\v x)}{f(\v x)}=\sum_{\v k \in C(\Delta - \v b)}a_{\v k}\v x^{\v k}.
\]
Then, for every
$\v k \in C(\Delta - \v b)$ and $s\ge1$  we have $a_{p^s\v k}\is a_{p^{s-1}\v k}\mod{p^s}$.
\end{corollary}

\begin{proof} It is sufficient to give a proof for a monomial $g(\v x) = \v x^\v v$, $\v v \in \Delta_\Z$.
Application of Theorem \ref{main2} with $\mu = \{\v v\}$, which is an open set due to our assumption on
$\Delta$, yields the congruence $a_{p^s\v k}\is \Lambda_\sigma a_{p^{s-1}\v k}\mod{p^s}$ with
$\Lambda_\sigma \in \Z_p^\times$.
Since $R=\Z_p$ we have that $f^\sigma_{\v v}=f_{\v v}$ and
hence $\Lambda_\sigma=1$ as in Corollary~\ref{special-delta}.
\end{proof}

In \cite{BHS18} the polytope $\Delta$ is a subset of the unit hypercube in $\R^n$,
hence the conditions of Corollary~\ref{b-h-s-congruences} are satisfied.

\section{Semi-simple decomposition}

Let $\tilde n \le n$ be the dimension of $\Delta$. For $0 \le l \le \tilde n$ let $\mu^{(l)} \subset \Delta$
be the complement of the union of faces of codimension $>l$; this is an open set in the topology
defined in Proposition~\ref{topology}. The inclusions
\[
\Delta^\circ = \mu^{(0)} \subset \mu^{(1)} \subset \ldots  \subset \mu^{(\tilde n)} = \Delta
\]
give rise to a filtration on the module of regular functions given by
\be{weight-filtration}
\Omega_f(\mu^{(0)}) \subset \Omega_f(\mu^{(1)}) \subset \ldots  \subset \Omega_f(\Delta) = \Omega_f.
\ee
Note that this filtration is preserved by the connection and its $p$-adic completion is preserved
by the Cartier operator, i.e. 
$\cartier: \hat \Omega_f^n(\mu^{(l)}) \to \hat \Omega_{f^\sigma}^n(\mu^{(l)})$
for each $l$ (see Proposition~\ref{topology}). We quotient the $p$-adic completions
by formally exact forms and obtain  
\[
Q_f(\mu^{(0)}) \subset Q_f(\mu^{(1)}) \subset \ldots  \subset Q_f(\Delta).
\]
Let $\beta_p^{(l)}$ be the Hasse--Witt matrix of $f$ relative to $\mu^{(l)}$. We shall
call $\beta_p^{(0)}$ simply \emph{the Hasse--Witt matrix of $f$}.
The following fact is a straightforward corollary of the congruences stated in
Theorem~\ref{main1}. As in this theorem, we assume that
$\cap_{s \ge 0} p^s R=\{ 0 \}$ and $R$ is $p$-adically complete. Recall that in
this case an element is invertible if and only if it is invertible modulo $p$.
Note also that for any face $\eta \subset \Delta$ the Newton polytope of the
restriction $f|_\eta$ is given by $\eta$.

\begin{theorem}\label{semi-simple}Assume that the coefficients of $f$ at all vertices of $\Delta$ 
are units in $R$. Let $l \ge 1$. The matrix $\beta_{p}^{(l)}$ is invertible if and only if $\beta_p^{(l-1)}$
and the Hasse--Witt matrices of all restrictions $f|_\eta$ to the faces $\eta \subset \Delta$ of
codimension $l$ are invertible. If this is the case, one has the following decomposition of the quotient crystal 
\[
Q_f(\mu^{(l)}) / Q_f(\mu^{(l-1)}) = \underset{\begin{tiny}\bal &\text{faces } \eta \subset \Delta \\ &{\rm codim} 
 (\eta) = l \eal \end{tiny}} \bigoplus Q_{f|_\eta}(\eta^\circ),
\]
where $\eta^\circ$ is the interior of the face $\eta$, and we make the convention that an interior of a vertex is the vertex itself.
\end{theorem}

\begin{proof} We write $\mu^{(\ell)} = \eta_1^\circ \cup \eta_2^\circ \cup \ldots \cup 
\mu^{(\ell-1)}$, where $\eta_1,\eta_2,\ldots$ are all faces of
$\Delta$ of codimension $\ell$, and claim that for any $m \ge 1$ matrices $\beta_m(\mu^{(\ell)})$ have the following block structure
\[
\beta_m(\mu^{(\ell)}) \= \begin{pmatrix} \beta_m(\eta_1^\circ) & 0 & 0 & \ldots & * \\
0 & \beta_m(\eta_2^\circ) & 0 & \ldots & * \\ 
&&&\ldots& *\\
0 & 0 & 0 & \ldots & \beta_m(\mu^{(\ell-1)})\\
\end{pmatrix}
\]
with diagonal blocks corresponding to all faces of codimension $\ell$ and $\mu^{(\ell-1)}$ and possibly non-zero off-diagonal blocks only in the last column. This claim follows from the following observation: if $\eta \subsetneq \Delta$ is a face, $\v v \in \eta$ and $\v u \not\in \eta$, then the coefficients of $m \v v - \v u$ in $f(\v x)^{m-1}$ is zero. Indeed, choose a linear functional $\kappa: \R^n \to \R$ such that $\kappa|_{\eta} \equiv c$, and $\kappa(\Delta) \subset \R_{\le c}$ and $\kappa(\v u)<c$ for some $c \in \R$. Then $\kappa(m \v v - \v u) = m c - \kappa(u) > (m-1)c$ and therefore $m \v v - \v u \not\in(m-1)\Delta$.  

Taking $m=p$ we see that the Hasse--Witt matrix $\beta_p(\mu^{\ell})$ is invertible if and only if all 
$\beta_p(\eta_i^\circ)$ and $\beta_p(\mu^{\ell-1})$ are invertible. Since the above mentioned block 
structure is preserved under taking the inverse, by the congruences in Theorem~\ref{main1} matrices 
$\Lambda_\sigma$ and $N_{\delta}$ for $\mu^{(\ell)}$ have the same block structure and the direct 
sum decomposition of the quotient crystal follows immediately.

\end{proof}

\begin{remark} Corollaries~\ref{special-delta} and~\ref{b-h-s-congruences} deal with
the situation when the only lattice points in the Newton polytope $\Delta$ are its vertices.
In this case filtration~\eqref{weight-filtration} has only one step 
(the set $\mu^{(\tilde n - 1)}_\Z$ is empty) and, assuming that the coefficients of 
$f(\v x)=\sum_{\v v} f_\v v \v x^\v v$ at all vertices are units in $R$,
Theorem~\ref{semi-simple} states that the unit-root crystal $Q_f=Q_f(\Delta)$ is a direct
sum of crystals of rank~1. 
\end{remark}

Let us mention that in the regular case (i.e. when $f$ is $\Delta$-regular)
under the identification of the Dwork module
$\dwork = \Omega_f/d\Omega_f$ with the cohomology group $H^n(\T^n \setminus Z_f)$
(after tensoring with the field of fractions of $R$) the image of the
filtration~\eqref{weight-filtration} is the weight filtration of the respective mixed
Hodge structure (see~\cite[Theorem 8.2]{Ba93}). 
Theorem~\ref{semi-simple} thus gives a semi-simple decomposition of unit-root crystals corresponding
to the graded pieces of the weight filtration.

\bigskip

\section{Example}
Consider $f=y^2+tx^3+xy+x$ and $R$ the $p$-adic completion of $\Z_p[t]$. 
We have the following sets of exponent vectors with $u_0=1$,
\begin{enumerate}
\item $\mu^{(2)}_\Z=\{(0,2),(1,0),(3,0),(2,0),(1,1)\}=\Delta_\Z$
\item $\mu^{(1)}_\Z=\{(2,0),(1,1)\}$
\item $\mu^{(0)}_\Z=\{(1,1)\}$
\end{enumerate}
where the $\mu^{(l)}$ are defined in the previous section. The ordering
of the exponent vectors in $\Delta$ is chosen in decreasing filtration order.
Using this ordered basis a straightforward calculation shows that for odd $m$,
\[
\beta_m=\begin{pmatrix}
1 & 0 & 0 & 0 & h_m(t)\\
0 & 1 & 0 & 0 & {1\over2}h_m(t)\\
0 & 0 & t^{m-1} & 0 & {1\over 2t}h_m(t)\\
0 & 0 & 0 & \left({m-1\atop {m-1\over2}}\right)t^{{m-1\over2}} & g_m(t)\\
0 & 0 & 0 & 0 & f_m(t)
\end{pmatrix},
\]
where
\[f_m(t)=\sum_{k\ge0}{m-1\choose 4k}{4k\choose 2k}{2k\choose k}t^k,\quad
g_m(t)=\sum_{k\ge0}{m-1\choose 4k+1}{4k+1\choose 2k}{2k\choose k}t^k,
\]
\[
h_m(t)=\sum_{k\ge1}{m-1\choose 4k-1}{4k-1\choose 2k}{2k\choose k}t^{k}.
\]
For the invertibility of $\beta_p$ we extend $R$ to be 
the $p$-adic completion of $\Z_p[t,(t f_p(t))^{-1}]$. Then $\beta_{p^s}
(\beta_{p^{s-1}}^\sigma)^{-1}$ reads
\[
\begin{pmatrix}
1 & 0 & 0 & 0 & (h_{p^s}(t)-h_{p^{s-1}}(t^p))/f_{p^{s-1}}(t^p)\\
0 & 1 & 0 & 0 & {1\over2}(h_{p^s}(t)-h_{p^{s-1}}(t^p))/f_{p^{s-1}}(t^p)\\
0 & 0 & t^{p-1} & 0 & {1\over2t}(h_{p^s}(t)-h_{p^{s-1}}(t^p))/f_{p^{s-1}}(t^p)\\
0 & 0 & 0 & c(p,s)t^{(p-1)/2} & (g_{p^s}-c(p,s)t^{(p-1)/2}g_{p^{s-1}}(t^p))
/f_{p^{s-1}}(t^p)\\
0 & 0 & 0 & 0 & f_{p^s}(t)/f_{p^{s-1}}(t^p)
\end{pmatrix}
\]
where $c(p,s)={p^s-1\choose (p^s-1)/2}{p^{s-1}-1\choose(p^{s-1}-1)/2}^{-1}$.
We now take the limit as $s\to\infty$. It is not hard to derive that
$c(p,s)\is(-1)^{(p-1)/2}\mod{p^s}$. Also, using ${p^s-1\choose m}
\is {p^{s-1}-1\choose \lfloor m/p\rfloor}\mod{p^s}$, one easily
shows that 
\[
h_{p^s}(t)-h_{p^{s-1}}(t^p)+{1\over2}(f_{p^s}(t)-f_{p^{s-1}}(t^p))\is0\mod{p^s}.
\] 
Finally, experiment shows that 
\[
{f_{p^s}(t)\over f_{p^{s-1}}(t^p)}\is {F(t)\over F(t^p)}\mod{p^s},
\]
where $F(t)=\,_2F_1(1/4,3/4,1|64t)$ and
\[
{g_{p^s}(t)-(-t)^{(p-1)/2}g_{p^{s-1}}(t^p)\over f_{p^{s-1}}(t^p)}
\is {G(t)-(-t)^{(p-1)/2}G(t^p)\over F(t^p)}\mod{p^s},
\]
where $G(t)=\,_3F_2(5/4,3/4,1/2;3/2,1|64t)$.
Putting everything together we find the limit
\[
\Lambda=\begin{pmatrix}
1 & 0 & 0 & 0 & {-1\over2}\left(\frac{F(t)}{F(t^p)}-1\right)\\
0 & 1 & 0 & 0 & {-1\over4}\left(\frac{F(t)}{F(t^p)}-1\right)\\
0 & 0 & t^{p-1} & 0 & {-1\over4t}\left(\frac{F(t)}{F(t^p)}-1\right)\\
0 & 0 & 0 & (-t)^{(p-1)/2} & \frac{G(t)-(-t)^{(p-1)/2}G(t^p)}{F(t^p)}\\
0 & 0 & 0 & 0 & \frac{F(t)}{F(t^p)}
\end{pmatrix}
\]
We make a few observations. 
\begin{itemize}
\item[(i)] $(1,1,t,0,1)\Lambda=(1,1,t^p,0,1)$
\item[(ii)] $(2,0,0,0,1)\Lambda=(2,0,0,0,1)$
\item[(iii)] $(0,1,3t,0,1)\Lambda=(0,1,3t^p,0,1).$
\end{itemize}
These equalities imply that in $Q_f$ we have 
\[
\cartier\left(\frac{f_i}{f}\right)=\frac{f_i^\sigma}{f^\sigma}\qquad
i=0,1,2,
\]
where $f_0=f,f_1=x\frac{\partial f}{\partial x},f_2=y\frac{\partial f}{\partial y}$.
This is a general phenomenon, as shown in the following theorem. 

\begin{proposition}\label{directsum}
Let $f(\v x) = f(x_1,\ldots,x_n)$ be a Laurent polynomial with coefficients in a characteristic 
zero ring $R$ such that at least one vertex coefficient of $f$ is a unit modulo $p$.
Let $f_0=f$ and $f_i=x_i\frac{\partial f}{\partial x_i}$ for $i=1,\ldots,n$. 
Then for all $i$,
\[
\cartier\left(\frac{f_i}{f}\right)\is\frac{f_i^\sigma}{f^\sigma}
\mod{p \, d\hat\Omega_{f^\sigma}}.
\]
\end{proposition}

\begin{proof}
The case $i=0$ comes down to $\cartier(1)=1$, which is trivial.
So let $i>0$. As earlier, we will use the notation $\theta_i = x_i \frac{\partial}{\partial x_i}$. 
Notice that 

\begin{eqnarray*}
\cartier\left(\frac{f_i}{f}\right)&=&
\cartier(\theta_i (\log f(\v x)))\\
&=&p \theta_i \left(\cartier(\log f(\v x))\right)\\
&=& \theta_i \left(\cartier(\log f(\v x)^p)\right)\\
&=& \theta_i \left(\cartier(\log f^\sigma(\v x^p))+
\cartier \left(\log\left(1+p\frac{G(\v x)}{f^\sigma(\v x^p)}\right)\right)\right),
\end{eqnarray*}
where $pG(\v x)=f(\v x)^p-f^\sigma(\v x^p)$. Observe that $\cartier(\log f^\sigma(\v x^p))
=\log f^\sigma(\v x)$. Power series expansion of the $\log$ in
\[
\cartier\left(\log\left(1+p\frac{G(\v x)}{f^\sigma(\v x^p)}\right)\right)
\]
gives us
\[
-\sum_{r\ge1}\frac{(-p)^r}{r}\frac{\cartier(G(\v x)^r)}{f^\sigma(\v x)^r}.
\]
Combining these evaluations gives us the final result
\[
\cartier\left(\frac{f_i}{f}\right)=
\frac{f^\sigma_i}{f^\sigma}(\v x)
-\sum_{r\ge1}\theta_i\left(\frac{(-p)^r}{r!}(r-1)!\frac{\cartier(G(\v x)^r)}
{f^\sigma(\v x)^r}\right).
\]
Clearly the latter summation belongs to $p \, d\hat\Omega_{f^\sigma}$ when $p>2$. 
\end{proof}

We now determine the limit of 
\[
\theta(\beta_{p^s})\beta_{p^s}^{-1}=
\begin{pmatrix}
0 & 0 & 0 & 0 & \frac{\theta h_{p^s}(t)}{f_{p^s}(t)}\\
0 & 0 & 0 & 0 & \frac{\theta h_{p^s}(t)}{2f_{p^s}(t)}\\
0 & 0 & p^s-1 & 0 & \frac{\theta h_{p^s}(t)-p^sh_{p^s}}{2tf_{p^s}(t)}\\
0 & 0 & 0 & {1\over2}(p^s-1) & \frac{2\theta g_{p^s}(t)-(p^s-1)g_{p^s}(t)}{2f_{p^s}(t)}\\
0 & 0 & 0 & 0 & \frac{\theta f_{p^s}(t)}{f_{p^s}(t)}
\end{pmatrix}
\]
where $\theta=t\frac{d}{dt}$. 
Some experiment suggests the following congruences
\[
\frac{\theta f_{p^s}(t)}{f_{p^s}(t)}\is \frac{\theta F(t)}{F(t)}\mod{p^s},
\quad
\frac{\theta h_{p^s}(t)}{f_{p^s}(t)}\is\frac{-1}{2}\frac{\theta F(t)}{F(t)}\mod{p^s}
\]
and
\[
\frac{2\theta g_{p^s}(t)+g_{p^s}(t)}{f_{p^s}(t)}
\is -\frac{4\theta F(t)+F(t)}{F(t)}\mod{p^s}.
\]
This yields the limit matrix
\[
N=\begin{pmatrix}
0 & 0 & 0 & 0 & {-1\over2}\frac{\theta F(t)}{F(t)}\\
0 & 0 & 0 & 0 & {-1\over4}\frac{\theta F(t)}{F(t)}\\
0 & 0 & -1 & 0 & {-1\over 4t}\frac{\theta F(t)}{F(t)}\\
0 & 0 & 0 & {-1\over2} & -2\frac{\theta F(t)}{F(t)}-{1\over2}\\
0 & 0 & 0 & 0 & \frac{\theta F(t)}{F(t)}
\end{pmatrix}
\]
From this limit matrix it easily follows that $\frac{f_i}{f},
i=0,1,2$ are horizontal in $Q_f$, that is they are annihilated by $\theta$. This is a general phenomenon.

\begin{proposition}\label{horizontal-fi-over-f}
Let notations be as in Proposition~\ref{directsum} and $\delta$ be a derivation on
$R$. Then we have
\[
\delta\left(\frac{f_i}{f}\right)\is0 \mod {d\Omega_f}.
\]
\end{proposition}

\begin{proof}
The proof is immediate,
\[
\delta\left(\frac{f_i}{f}\right) = \frac{\delta(f)_i \cdot f - f_i \cdot \delta(f)}{f^2}
 = \theta_i \Bigl( \frac{\delta(f)}{f}\Bigr).
\]
\end{proof}

Finally, getting back to our example, we mention the matrix
\[
Y=\begin{pmatrix}
1 & 0 & 0 & 0 & {-1\over2}F(t)\\
0 & 1 & 0 & 0 & {-1\over4}F(t)\\
0 & 0 & t^{-1} & 0 & {-1\over 4t}{F(t)}\\
0 & 0 & 0 & t^{-1/2} & G(t)\\
0 & 0 & 0 & 0 & F(t)
\end{pmatrix}.
\]
This is a fundamental solution matrix of the system of first order equations
\[
\theta\v y(t)=N\v y(t),
\]
where $\v y$ is a column vector of $5$ unknown functions in $t$.

\bigskip

Note that Propositions~\ref{directsum} and~\ref{horizontal-fi-over-f} have nothing to
do with the unit-root crystal $Q_f$: their statements hold modulo
$d \widehat \Omega_{f^\sigma}$ and $d \Omega_f$ respectively and not just modulo
formal derivatives. These propositions show that $\oplus_{i=0}^n R \frac{f_i}f$ is a
subcrystal in the completed Dwork crystal $\widehat \dwork = \widehat \Omega_f / d \widehat \Omega_f$,
on which the Cartier operator acts as the identity. In the geometric situation mentioned
at the end of the introductory section, this subcrystal should correspond to the embedding
of $H^n_{dR}(\T^n)$ into $H^n_{dR}(\T^n \setminus Z_f) \cong \dwork$.

\bigskip

\section*{Acknowledgement}
The authors are grateful to Alan Adolphson for his comments on the manuscript.

We would like to thank to the MATRIX institute in Creswick, Max Planck Institute and Hausdorff Research Institute for Mathematics in Bonn, the University of Utrecht and Institute of Mathematics of the Polish Academy of Sciences in Warsaw, where collaboration on this project partly took place. The second author also thanks to the Mathematical Sciences Research Institute in Berkeley.

\appendix
\section{Point counting and an alternative construction of the Cartier operator}
Suppose that $R=\Z_q$ where $q=p^a$ and $\sigma$ is the standard $p$th power Frobenius lift 
satisfying $\sigma^a=id$. Then the $a$th iteration of the Cartier operator 
$\scrC_q:=\cartier^a$ maps $\hat\Omega_f$ to itself. It follows from the 
estimate~\eqref{cartier-entries-estimate} that modulo every power $p^s$ the image of
$\cartier$ has finite rank. By this reason the trace of $\scrC_q$ is a well defined $p$-adic value.
In this section we will prove
the following 

\begin{theorem}\label{pointcount}
The trace of $(q^s-1)^n \times \scrC_q^s$ on $\hat \Omega_f$ equals 
the number of points on $\T^n \setminus Z_f$ with coordinates $x_1,\ldots,x_n\in\F^\times_{q^s}$.
\end{theorem}

\begin{remark} Since $\scrC_q$ is divisible by $q$ on the submodule of formal derivatives $U_f \subset \widehat \Omega_f$, we conclude from Theorem~\ref{pointcount} that on the quotient $Q_f = \widehat\Omega_f / U_f$ one has
\[
\Tr(\scrC_q^s | Q_f) \equiv 1+(-1)^{n+1} \#Z_f(\F_{q^s}) \; \mod {q^s}. 
\]
The term $1$ on the right corresponds to the eigenvector $1 \in \Omega_f$, which has eigenvalue $1$.
If the Hasse--Witt matrix $\beta_p(\Delta)$ is invertible, then $Q_f$ is a $\Z_q$-vector space of finite dimension and it follows from the above congruence that the polynomial
\[
\det(1- T \scrC_q | Q_f) / (1-T) \in \Z_q[T]
\]
raised to the power $(-1)^{n}$
is the unit-root part of the zeta function of the hypersuface $Z_f \subset \T^n$ over $\F_q$. Note that in our standard basis in $Q_f$ (that is, images of $\omega_\v u \in \Omega_f, \v u \in \Delta_\Z$) the operator $\scrC_q$ is given by the transpose of 
\[
\Lambda \, \sigma(\Lambda) \, \ldots \, \sigma^{a-1}(\Lambda),
\]
where $\Lambda=\Lambda_\sigma$ is the matrix from Theorem~\ref{main1}.
\end{remark}

\bigskip

We will use a resolution of the module $\Omega_f$. This construction ties our crystals with the exponential modules in the literature, e.g. in~\cite{Ba93},
and exhibits a natural lift of our Cartier operator which possesses nice properties and hence might be
useful on its own. With the lifted Cartier operator, the point counting can be done using a version of
Dwork's trace formula, which is now a standard technique in $p$-adic analysis. 

From now on $R$ is a characteristic zero ring, we will impose more assumptions when needed. Let us introduce
the auxiliary variable $x_0$ and define the subring 
\[
R[\Delta] \subset R[x_0,x_1^{\pm 1},\ldots,x_n^{\pm 1}]
\]
as the span of monomials $\v X^\v u = x_0^{u_0} \ldots x_n^{u_n}$ with $\v u \in C(\Delta)_\Z$. 
We remind the reader that throughout the paper we denoted $\v x^{\v u} = x_1^{u_1} \ldots x_n^{u_n}$
for $\v u=(u_0,\ldots,u_n)$, so now we shall use the capital letter for $\v X^\v u = x_0^{u_0} \v x^{\v u}$.
We denote $F(\v X) = x_0 f(\v x) \in R[\Delta]$. The operations 
\[
\scrD_{i, f} \= x_i \frac{\partial}{\partial x_i} + x_i \frac{\partial F(\v X)}{\partial x_i} : R[\Delta] \to R[\Delta], \quad i = 0,\ldots, n
\]
are called \emph{twisted derivatives}. Formally, we have 
$\scrD_{i,f} = e^{-F} \circ \theta_i \circ e^F$, where $\theta_i = x_i \frac{\partial}{\partial x_i}$.
Twisted derivatives commute with each other.      

Let $R[\Delta]^+$ be the free $R$-module generated by $\v X^\v u$ with $\v u \in C(\Delta)_Z^+$.
It is an ideal in $R[\Delta]$ and twisted derivatives preserve $R[\Delta]^+$. The \emph{Laplace transform}
is the $R$-linear map $\scrR : R[\Delta]^+ \to \Omega_f$ given by
\[\v X^{\v u} \mapsto (-1)^{u_0} (u_0-1)! \frac{\v x ^\v u}{f(\v x)^{u_0}} \=  (-1)^{u_0}\omega_{\v u}.
\]
This Laplace transform was basically defined in~\cite[p.244]{Ka68} and~\cite[\S 7]{Ba93}. 
See also our Remark~\ref{whyfactorials}. It is clear that $\scrR$ is surjective. 

\begin{proposition}\label{resolving-omega-f} The kernel of the Laplace transform is given by
$\scrD_{0,f}(R[\Delta]^+)$. Under the induced isomorphism
\[
\scrR : R[\Delta]^+/ \scrD_{0,f}(R[\Delta]^+) \overset{\sim}\to \Omega_f 
\] 
the twisted defivative $\scrD_{i,f}$ corresponds to the usual derivative $\theta_i$ for each $1 \le i \le n$.
\end{proposition} 
\begin{proof} Under the Laplace transform the elements $\scrD_{0,f}
(\v X^\v u) = u_0 \v X^{\v u} + x_0 f(\v x) \v X^{\v u}$ are mapped to 
\[
(-1)^{u_0}u_0! \frac{\v x ^\v u}{f(\v x)^{u_0}} + (-1)^{u_0+1} u_0!  \frac{\v x ^\v u f(\v x)}{f(\v x)^{u_0+1}} \= 0.
\]  
It is clear that these elements generate all relations in $\Omega_f$ and therefore they span the kernel of $\scrR$. 

Let $1 \le i \le n$. Since twisted derivatives commute, $\scrD_{i,f}$ maps ${\rm Ker}(\scrR)={\rm Im}(\scrD_{0,f})$ to itself. The fact that the induced map on $\Omega_f$ coincides with $\theta_i$ can be easily checked on monomials:
\[\bal
\scrR \left(\scrD_{i,f} (\v X^\v u) \right) &\= \scrR \left( u_i \v X^\v u + \theta_i F(\v X) \v X^\v u \right)\= (-1)^{u_0} (u_0-1)! \left(  u_i \frac{\v x^{\v u}}{f(\v x)^{u_0}} - u_0 \frac{\v x^{\v u} \theta_i f(\v x)}{f(\v x)^{u_0+1}}\right)\\ & \= (-1)^{u_0} \theta_i \omega_\v u = \theta_i \scrR(\v X^\v u).
\eal\]
\end{proof} 
 
\begin{corollary}\label{dwork-deRham} $R[\Delta]^+/\sum_{i=0}^{n} \scrD_{i,f}(R[\Delta]^+) \cong \dwork$.
\end{corollary}

We would like to remark that in~\cite[Theorem 7.13]{Ba93} the quotient module on the left in this corollary was identified with $H^n_{dR}(\T^n \setminus Z_f)$ under the condition that $R$ is a field and $f(\v x)$ is $\Delta$-regular. At the end of the introductory section we mentioned the relation between Dwork modules and de Rham cohomology having in mind Corollary~\ref{dwork-deRham}.

To define the Cartier operator, we turn on our usual assumptions that $\cap_s p^s R = \{0\}$ and $R$ is $p$-adically complete. Let $R\lb \Delta \rb$ be the ring of formal power series with coefficients in $R$ and support in $C(\Delta)$. The $p$-adic completion $\widehat{R[\Delta]} = \underset{s} \lim R[\Delta]/p^s R[\Delta]$ consists of power series with infinitely growing $p$-adic valuation of coefficients:  
\[
\widehat{R[\Delta]} = \left\{ \sum_{\v u \in C(\Delta)_\Z} a_\v u \v X^\v u \;:\; a_\v u \in R, \quad \ord_p(a_\v u) \to \infty \hbox{ as } u_0 \to \infty \right\} \subset R\lb \Delta \rb. 
\]
We denote by $\widehat{R[\Delta]}^+$ the ideal of power series with zero constant term ($a_\v 0 = 0$). It follows from Proposition~\ref{resolving-omega-f} that $\hat\Omega_f \cong \widehat{R[\Delta]}^+/\scrD_{0,f}\Bigl(\widehat{R[\Delta]}^+\Bigr)$.
 
\begin{theorem}\label{cartier-coincide-thm} Consider the operator on power series given by 
\[
V_p \left( \sum_{\v u} a_{\v u} \v X^{\v u} \right) \=  \sum_{\v u} (-p)^{u_0} a_{p \v u} \, \v X^{\v u}.
\]
Let $p>2$. For every $p$th power Frobenius lift $\sigma: R \to R$, the operator 
\[
\scrV_{\sigma} = e^{-F^\sigma} \circ V_p \circ e^F
\]
maps $R\lb\Delta\rb$ to itself. Operator $\scrV_\sigma$ preserves $R\lb \Delta\rb^+$ and it is divisible by $p$ on this submodule. The following commutation relation with twisted derivatives
\be{p-commutation-cartier-lift}
\scrD_{i,f^\sigma} \circ \scrV_\sigma = p \, \scrV_\sigma \circ \scrD_{i,f}
\ee
holds for each $0 \le i \le n$. 

The operator $\scrV_\sigma$ preserves $\widehat{R[\Delta]}$ and the induced map $p^{-1} \scrV_\sigma: \widehat\Omega_f \to \widehat \Omega_{f^\sigma}$ coincides with the Cartier operator $\cartier: \widehat\Omega_f \to \widehat \Omega_{f^\sigma}$ constructed in Section~\ref{sec:cartier}.
\end{theorem}

\begin{proof} Fix any $\v u \in C(\Delta)_\Z^+$ and denote $\nu_0=p\lceil \frac{u_0}{p}\rceil - u_0$. Observe that
\[\bal
V_p \Bigl( e^{x_0 f(\v x) }\v X^\v u\Bigr) &\= \sum_{n \ge 0,\; p\,|(n+u_0)} \frac{(- p x_0)^{\frac{n+u_0}p} \cartier(f(\v x)^n \v x^\v u)}{n!} \\
&\= \sum_{m \ge 0} \frac{(- p x_0)^{m+\lceil \frac{u_0}{p}\rceil} \cartier(f(\v x)^{p m + \nu_0} \v x^\v u)}{(p m + \nu_0)!}
\eal\]
and therefore 
\[\bal
\scrV_\sigma \v X^{\v u} &\= \sum_{k \ge 0} \frac{(-x_0)^{k} f^\sigma(\v x)^k}{k!} \sum_{m \ge 0} \frac{(- p x_0)^{m+\lceil \frac{u_0}{p}\rceil} \cartier(f(\v x)^{p m + \nu_0} \v x^\v u)}{(p m + \nu_0)!}\\
&\=  \sum_{r \ge 0} (- p x_0)^{r + \lceil \frac{u_0}{p}\rceil} \, \cartier \Bigl( \v x^\v u \sum_{m+k=r} \frac{f(\v x)^{p m + \nu_0} f^\sigma(\v x^p)^k}{(pm+\nu_0)! \, k! p^k}\Bigr)\\
&\= \sum_{r \ge 0} (-p x_0)^{r + \lceil \frac{u_0}p\rceil} \cartier \Bigl( \v x^\v u \sum_{s = 0}^r  \frac{G(\v x)^{r-s} f(\v x)^{ps+\nu_0}}{(r-s)!} \gamma_{p s + \nu_0} \Bigr)\\ 
\eal\]
where we substituted $f^\sigma(\v x^p)=f(\v x)^p + p G(\v x)$ and recognised the sums
\[
\gamma_{p s + \nu_0} =  \sum_{m=0}^{s} \frac{1}{(pm+\nu_0)! \, (s-m)!p^{s-m}}\\
\]
as coefficients of the \emph{Dwork exponential} $e^{x_0+\frac{x_0^p}{p}} = \sum_{n \ge 0} \gamma_n x_0^n$. The following standard estimate of their $p$-adic order 
\be{dwork-exponential-estimate}
\ord_p(\gamma_n) \ge \bigl(\frac{p-1}{p^2}-\frac{1}{p-1}\bigr)n
\ee
implies that the matrix coefficients given by $\scrV_\sigma \v X^\v u = \sum_{\v v} \scrF_{\v u,\v v} \v X^\v v$ have $p$-adic valuations bounded by
\[\bal
\ord_p (\scrF_{\v u,\v v}) &\ge r + \lceil \frac{u_0}p\rceil + \underset{0 \le s \le r}\min \Bigl( (-\frac1{p-1})(r-s) + \bigl( \frac{p-1}{p^2} - \frac{1}{p-1}\bigr)\bigl(ps+\nu_0 \bigr) \Bigr) \\
&\= r + \lceil \frac{u_0}p\rceil - \frac{r}{p-1} +  \underset{0 \le s \le r}\min \Bigl(- \frac{s}p  \Bigr) + \bigl( \frac{p-1}{p^2} - \frac{1}{p-1}\bigr) \bigr(p\lceil \frac{u_0}p \rceil - u_0\bigl)  \\
&\= \bigl( \frac{p-1}{p} - \frac1{p-1}\bigr) \bigl( r + \lceil \frac{u_0}p\rceil  \bigr) - \bigl( \frac{p-1}{p^2} - \frac{1}{p-1} \bigr)u_0 \\
& \= \bigl( \frac{p-1}{p} - \frac1{p-1}\bigr) v_0 + \bigl(\frac{1}{p-1} - \frac{p-1}{p^2}\bigr)u_0,
\eal\] 
where we used $r = v_0 - \lceil \frac{u_0}{p}\rceil$. Since this valuation is non-negative, we conclude that all $\scrF_{\v u, \v v} \in R$ and hence $\scrV_\sigma$ maps the module of formal series $R\lb \Delta \rb$ to itself. We also observe that $\scrV_\sigma$ is divisible by $p$ on $R\lb \Delta \rb^+$. Moreover, $\scrV_\sigma$ preserves $\widehat {R[\Delta]}$ because when $v_0 \to \infty$ we have $\ord_p(\scrF_{\v u, \v v} \in R) \to \infty$ uniformly in $u_0$.

\bigskip   

The commutation relation~\eqref{p-commutation-cartier-lift} follows immediately from $\theta_i \circ V_p = p \, \theta_i \circ V_p$ and the fact that $\scrD_{i,f}=e^{-F}\circ \theta_i\circ e^F$. Since $\widehat \Omega_f = \widehat{R[\Delta]}^+ / \scrD_{0,f}( \widehat{R[\Delta]}^+)$, we have a well defined induced map $p^{-1} \scrV_\sigma: \widehat\Omega_f \to \widehat\Omega_f^\sigma$. Let us show that this induced map coincides with the Cartier operator $\cartier$ defined in Proposition~\ref{cartier-def}. Since $\scrR(\v X^\v w)=(-1)^{u_0} \omega_{\v w}$, we consider 
\be{difference-of-two-cartier}\bal
 & p^{-1}  \scrV_\sigma \v X^\v u \- \sum_{\v v \in C(\Delta)_\Z^+} (-1)^{v_0+u_0} \, F_{\v u, \v v} \v X^\v v \\
 & \= \frac1p \sum_{r \ge 0} (-p x_0)^{r + \lceil \frac{u_0}p\rceil} 
 \cartier \Bigl( \v x^\v u \sum_{s = 0}^r  G(\v x)^{r-s} f(\v x)^{ps+\nu_0} 
 \frac{\gamma_{ps+\nu_0} - \delta_{s,0} \Gamma_p(u_0)}{(r-s)!}    \Bigr),
\eal\ee
where we used formula~\eqref{cartiermatrix-entries} and substitution 
$\frac{(u_0-1)!}{(\lceil \frac{u_0}{p}\rceil - 1)!}=(-1)^{u_0} p^{\lceil \frac{u_0}{p}\rceil - 1} 
\Gamma_p(u_0)$. It is easy to check that the difference~\eqref{difference-of-two-cartier} equals
\be{difference-of-two-cartier-as-derivative}
\scrD_{0,f^\sigma} \Bigl( \frac1p  \sum_{r \ge 0} (-p x_0)^{r + \lceil \frac{u_0}p\rceil} \cartier \bigl( \v x^\v u \sum_{s = 0}^r G(\v x)^{r-s} f(\v x)^{ps+\nu_0} \frac{ \mu_s }{(r-s)!} \bigr)\Bigr),
\ee
where the coefficients $\mu_s \in \Q$ are determined by the recurrence
\[
\bigl(s+\lceil \frac{u_0}p\rceil\bigr) \mu_{s} \= p^{-1} \, \mu_{s-1} + \gamma_{p s + \nu_0} - \Gamma_p(u_0) \delta_{s,0}.
\]
(The initial term $\mu_0$ is also determined by this formula and convention $\mu_{-1}=0$.) We claim that
\be{mu-padic-rate}
\ord_p(\mu_s) \ge \bigl(-\frac{1}{p} - \frac{1}{p-1}\bigr) s + 1 - \lceil \frac{u_0}p\rceil , 
\ee
and hence inside of $\scrD_{0,f^\sigma}(\ldots)$ in~\eqref{difference-of-two-cartier-as-derivative} the $p$-adic valuation of coefficients of the polynomial next to $x_0^{r + \lceil \frac{u_0}p\rceil}$ can be estimated from below as
\[\bal
-1 + r + \lceil \frac{u_0}p\rceil + &\underset{0 \le s \le r}\min \ord_p\bigl(\frac{\mu_s}{(r-s)!}\bigr) \ge r + \underset{0 \le s \le r}\min \bigl(\bigl(-\frac{1}{p} - \frac{1}{p-1}\bigr) s -\frac1{p-1} (r-s)\bigr)\\
&\= r + \bigl(-\frac{1}{p} - \frac{1}{p-1}\bigr)r \= \bigl(\frac{p-1}{p} - \frac{1}{p-1}\bigr)r.   
\eal\] 
Since this valuation is non-negative and grows infinitely as $r \to \infty$, we conclude that~\eqref{difference-of-two-cartier-as-derivative} belongs to $\scrD_{0,f^\sigma}\bigl(\widehat{R[\Delta]^+}\bigr)$ and therefore $p^{-1} \scrV_\sigma = \cartier$ on $\widehat \Omega_f$.

It only remains to prove~\eqref{mu-padic-rate}. For this purpose we consider $f(\v x)=1$ in~\eqref{difference-of-two-cartier} and~\eqref{difference-of-two-cartier-as-derivative}. In this case $\scrV_\sigma$ is multiplication by the Dwork exponential $e^{x_0 + p^{-1} x_0^p}$, followed by $\cartier$ and the substitution $x_0 \mapsto -p x_0$. We shall denote $\scrV_\sigma$ simply by $\scrV$ and $\scrD_{0,f^\sigma} = x_0 \frac{d}{d x_0}+x_0$ by $\scrD$. The equality of~\eqref{difference-of-two-cartier-as-derivative} and~\eqref{difference-of-two-cartier} can be written as
\[\bal
\scrD \bigl( p^{-1} \sum_{r \ge 0} (-p x_0)^{r + \lceil \frac{u_0}p\rceil}  \mu_r \bigr) &\= p^{-1} \scrV ( x_0^{u_0}) \- p^{-1} \Gamma_p(u_0) (- p x_0)^{\lceil \frac{u_0}p\rceil} \\
&\= p^{-1} \scrV \bigl( x_0^{u_0} - \Gamma_p(u_0) x_0^{p \lceil \frac{u_0}p\rceil} e^{- x_0 - p^{-1} x_0^p}\bigr).
\eal\] 
Note that $\scrD$ is invertible on $\Q_p\lb x_0 \rb^+ = x_0 \Q_p\lb x_0 \rb$, and hence the commutation relation $\scrV \scrD = p \scrD \scrV$ can be rewritten as $p^{-1} \scrD^{-1} \scrV =  \scrV \scrD^{-1}$. Applying $\scrD^{-1}$ to the last identity and using the commutation relation, we get
\be{mu-series-via-cartier}
p^{-1} \sum_{r \ge 0} (-p x_0)^{r + \lceil \frac{u_0}p\rceil}  \mu_r  \= \scrV \scrD^{-1} \bigl( x_0^{u_0} - \Gamma_p(u_0) x_0^{p \lceil \frac{u_0}p\rceil} e^{- x_0 - p^{-1} x_0^p}\bigr).
\ee
Let 
\[
L_{\alpha,\beta} = \{ \sum_{m \ge 1} a_m x_0^m \in x_0 \Q_p\lb x_0 \rb \;|\; \ord_p(a_m) \ge \alpha m + \beta \}.
\]
Let $\alpha_0 = \frac{p-1}{p^2}-\frac1{p-1}$ and $\alpha_1=\frac{p-1}{p}-\frac1{p-1}$. Note that~\eqref{mu-padic-rate} precisely means that the series in~\eqref{mu-series-via-cartier} belongs to $L_{\alpha_1,0}$. In order to demonstrate this fact, we first notice that for any $\alpha \ge \alpha_0$ and any $\beta$ we have $\scrV: L_{\alpha,\beta} \to L_{\alpha_1,\beta}$. Indeed, since $e^{x_0+p^{-1}x_0^p} \in L_{\alpha_0,0}$ we decompose $\scrV$ into three steps and check that
\[
L_{\alpha,\beta} \overset{\cdot e^{x_0+p^{-1}x_0^p}}\to L_{\min(\alpha_0,\alpha),\beta} = L_{\alpha_0,\beta} \overset{\cartier}\to L_{p \alpha_0,\beta} \overset{x_0 \mapsto -px_0}\to L_{p \alpha_0+1,\beta}=L_{\alpha_1,\beta}.
\]
In the view of~\eqref{mu-series-via-cartier}, it now suffices to show that
\be{integral-todo-estimate}
\scrD^{-1} \bigl( x_0^{u_0} - \Gamma_p(u_0) x_0^{p \lceil \frac{u_0}p\rceil} e^{- x_0 - p^{-1} x_0^p}\bigr) \in L_{\alpha_0,0}.
\ee
It is useful to observe that $\scrD (x_0^m) = m x_0^m + x_0^{m+1}$ and 
\[
\scrD(x_0^{m} e^{- x_0 - p^{-1} x_0^p})= e^{-x_0} x_0 \frac{d}{d x_0} (x_0^{m} e^{- p^{-1} x_0^p}) = (m x_0^m - x_0^{m+p}) e^{- x_0 - p^{-1} x_0^p}.
\]
Using these two rules one can easily check that 
\be{int-monomial}
\scrD \bigl( (-1)^{u_0-1}(u_0-1)!\sum_{m=0}^{u_0-1} \frac{(-x_0)^m}{m!}  \bigr) =  x_0^{u_0} 
\ee
and 
\be{int-twisted-monomial}
\scrD \bigl( (n-1)! p^{n-1} \sum_{m=0}^{n-1} \frac{x_0^{p m}}{m!\, p^m}  e^{- x_0 - p^{-1} x_0^p}\bigr) =  - x_0^{p n} e^{- x_0 - p^{-1} x_0^p}.
\ee
Note that under $\scrD(\ldots)$ the polynomial in~\eqref{int-monomial} has integral coefficients and the series in~\eqref{int-twisted-monomial} belongs to $L_{\alpha_0,0}$ if one cuts off its constant term. We shall use~\eqref{int-twisted-monomial} with $n = \lceil \frac{u_0}p\rceil$.  Since $(-1)^{u_0}(u_0-1)!=\Gamma_p(u_0)(\lceil \frac{u_0}p\rceil-1)!p^{\lceil \frac{u_0}p\rceil-1}$, we get
\[\bal
\scrD^{-1} \bigl( x_0^{u_0} - &\Gamma_p(u_0) x_0^{p \lceil \frac{u_0}p\rceil} e^{- x_0 - p^{-1} x_0^p}\bigr)
\\
&\= (-1)^{u_0-1}(u_0-1)! \Bigl( \sum_{m=0}^{u_0-1} \frac{(-x_0)^m}{m!} - \sum_{m=0}^{\lceil \frac{u_0}{p}\rceil-1} \frac{x_0^{p m}}{m!\, p^m}  e^{- x_0 - p^{-1} x_0^p} \Bigr).
\eal\] 
Note that the constant term of the series in the right-hand side vanishes, which means that we integrated~\eqref{integral-todo-estimate} in $x_0\Q_p\lb x_0 \rb$ explicitly. Since $\Gamma_p(u_0)$ is a $p$-adic integer and $\alpha_0<0$, this series belongs to $L_{\alpha_0,0}$ due to the remarks made after~\eqref{int-monomial} and~\eqref{int-twisted-monomial}. This completes our proof of~\eqref{mu-padic-rate}.
\end{proof}

\begin{remark} One can easily define a connection on $R[\Delta]$ in a way and it commutes with the twisted derivatives. Namely, for every derivation $\delta: R \to R$ we define its action on $R[\Delta]$ as
\[
\nabla_\delta := \delta + \delta F,
\] 
where the first summand simply means that the derivation $\delta$ is applied to the coefficients and the second one means multiplication by the polynomial $\delta F (\v X) = x_0 (\delta f)(\v x)$. Formally, one can write $\nabla_\delta = e^{-F} \cdot \delta \cdot e^F$. To see that $\nabla_\delta$ commutes with the twisted derivatives, recall that $\scrD_{i,f}=e^{-F} \cdot \theta_i \cdot e^F$ and note that $\delta$ and $\theta_i$ commute. Operations $\nabla_\delta$ preserve $R[\Delta]^+$ and descend to its quotients by the images of twisted derivatives, particularly to $\Omega_f$ and $\dwork$. It is easy to check that $\nabla_\delta$ acts on $\Omega_f$ as the natural extension of $\delta$ to rational functions, the operation which we simply denoted by the same letter $\delta$ earlier in this paper. 

\bigskip

Finally, observe that the operator $\scrV_\sigma = e^{-F^\sigma} \cdot V_p \cdot e^F$ defined in Theorem~\ref{cartier-coincide-thm} commutes with the connection operators. Namely, it is obvious that $V_p$ commutes with $\delta$ as operators on power series, and after twisting by exponentials we obtain 
\[
\scrV_\sigma \cdot \nabla_\delta = \nabla^\sigma_\delta \cdot \scrV_\sigma,
\] 
where $\nabla^\sigma_\delta = \delta + \delta F^\sigma = e^{-F^\sigma} \cdot \delta \cdot e^{F^\sigma}$. This observation turns quotients of $\widehat{R[\Delta]}^+$ by twisted derivatives into crystals. 
\end{remark}

From now on we consider $R=\Z_q$ with $q=p^a$. Here we have the standard $p$th power Frobenius lift $\sigma: \Z_q \to \Z_q$ which satisfies $\sigma^a=id$. Consider the operator on power series $Z_q\lb \Delta \rb$  given by 
\[
\scrV_q := e^{-F} \circ V_p^a \circ e^F .
\]
Below we compute the traces of powers of $\scrV_q$ using a few standard tricks in $p$-adic analysis, which are basically due to Dwork. 

\begin{remark}\label{on-traces-definition} The traces are well-defined $p$-adic numbers because modulo every $p^s$ the operator $\scrV_q$ has finite-dimensional image (see the $p$-adic estimate of the matrix entries in the proof of Theorem~\ref{cartier-coincide-thm}).  Note also that the traces only depend on the mod $p$ reduction of the polynomial $F(\v X)=x_0 f(\v x)$. Indeed, if $F'(\v X)-F(\v X)=p G(\v X)$ then the respective operators on power series are conjugate $\scrV_q' = e^{- p G} \circ \scrV_q \circ e^{p G}$ and modulo each power of $p$ this identity can be written using matrices of finite size.
\end{remark}

\begin{proposition}\label{trace-formula} For all $s \ge 1$ one has
\[
(1-q^s)^{n+1} \Tr \left(\scrV_q^s \;|\; \Z_q \lb \Delta\rb \right) \= q^s \; \# Z_f(\F_{q^s}) \- (q^s-1)^n. 
\]

\end{proposition}
\begin{proof} Let $\pi \in \overline{\Q_p}$ be a number satisfying $\pi^{p-1}=-p$. We will work with Laurent series with coefficients in $R=\Q_q(\pi)$ and support in the cone $C(\Delta)$. Let $\rho: R\lb \Delta \rb \to R\lb \Delta \rb$ be the operation given by $\rho(\v X^\v u)=\pi^{u_0} \v X^\v u$. Note that $V_p = \rho^{-1} \circ \cartier \circ \rho$ and $e^{\pm F} = \rho^{-1} \circ e^{\pm \pi F} \circ \rho$, and hence
\be{conjugate-cartier}
\rho \circ \scrV_q^s \circ \rho^{-1} = e^{-\pi F} \circ \scrC_q^{s} \circ e^{\pi F} \= \scrC_q^{s} \circ \Theta_{F,s}
\ee
where we used the power series 
\be{theta-F}
\Theta_{F,s}(\v X) := \exp\bigl( \pi F(\v X) - \pi F(\v X^{q^s})\bigr) = \sum_{u \in C(\Delta)_\Z} b_{\v u} \v X^\v u.
\ee
From~\eqref{conjugate-cartier} it is clear that $\Tr(\scrV_q^s)=\sum_{\v u \in C(\Delta)_\Z} b_{(q^s-1) \v u}$. Since 
\[
\sum_{x \in \Z_q: x^{q-1}=1} x^u \= \bcs q-1, &\mbox{ if } (q-1)|u, \\ 0, &\mbox{ otherwise }, \ecs
\]  
this trace can be computed by summation of values of~\eqref{theta-F} over tuples of Teichm\"uller units in $\Z_{q^s}$:
\be{trace-formula-1}
\sum_{\v X \in \Z_{q^s}^{n+1} \; : \; x_i^{q^s-1} = 1} \Theta_{F,s}(\v X) \= (q^s-1)^{n+1} \sum_{\v u \in C(\Delta)_\Z} b_{(q^s-1) \v u}  \= (q^s-1)^{n+1} \Tr(\scrV_q^s).
\ee

To evaluate the sum on the left, consider the \emph{Dwork exponential} $\theta_p(z)=\exp(\pi z - \pi z^p)$. This series has $p$-adic radius of convergence $> 1$ and $\zeta_p := \theta_p(1) = 1 + \pi \mod {\pi^2}$ is a $p$th root of unity. For $k \ge 1$, let $\theta_{p^k}(z)=\exp(\pi z - \pi z^{p^k}) = \prod_{i=1}^{k-1}\theta_p(z^{p^i})$. The additive character $\psi_k: \F_{p^k} \to \Q_p(\pi)^\times$ given by $\psi_k(\bar x) = \zeta_p^{\Tr_{\F_{p^k}/\F_p}(\bar x)}$ is related to the Dwork exponential via $\psi_k(\bar x) = \theta_{p^k}(\Teich(\bar x))$.

Write $F(\v X) = x_0 f(\v x) = \sum a_\v u \v X^\v u$ and let $\bar F(\v X) = \sum \bar a_\v u \v X^\v u$ with $\bar a_\v u \in \F_q$ be the reduction of $F$ modulo $p$. Denote $a_{\v u}' = \Teich(\bar a_\v u)$ and $F'(\v X)=\sum a_\v u' \v X^\v u$. For any vector $\bar X =(\bar x_0,\ldots, \bar x_n) \in \F_{q^s}^{n+1}$ we have 
\[
\psi_{as}(\bar F(\bar{ \v X})) = \psi_{as}(\sum \bar a_\v u \bar {\v X}^\v u ) = \prod \theta_{q^s}(a_\v u' \Teich(\bar{\v X})^\v u) = \Theta_{F'}(\Teich(\bar{\v X})).
\]
Therefore the left-hand sum in~\eqref{trace-formula-1} for $F'$ can be evaluated as
\be{trace-formula-2}\bal
&\sum_{\v X \in \Z_{q^s}^{n+1} \; : \; x_i^{q^s-1} = 1} \Theta_{F',s}(\v X) \= \sum_{\bar{\v X} \in (F_{q^s}^\times)^{n+1}} \psi_{as}(\bar F(\bar{\v X})) \= \sum_{\bar {\v x} \in (F_{q^s}^\times)^{n}} \sum_{x_0 \in F_{q^s}}  \psi_{as}(\bar x_0 f(\bar{\v x}))\\
&\= \sum_{\bar {\v x} \in (F_{q^s}^\times)^{n}} \bcs q^s-1, &\mbox{ if } \bar f(\bar{\v x}) = 0 \\ -1 , &\mbox{ if  } \bar f(\bar{\v x}) \ne 0 \ecs \quad \= q^s \, \# Z_f(\F_{q^s}) \- (q^s-1)^n. 
\eal\ee
By Remark~\ref{on-traces-definition}, since $F'\is F \mod p$ traces of powers of $\scrV_q'=e^{-F'} \circ V_p^a \circ e^{F'}$ and $\scrV_q = e^{-F} \circ V_p^a \circ e^{F}$ are equal. Hence our claim follows from~\eqref{trace-formula-1} and~\eqref{trace-formula-2}.
\end{proof}

\begin{proof}[Proof of Theorem~\ref{pointcount}.] By Theorem~\ref{cartier-coincide-thm}, we have
\be{resolution-traces}
q^s \, \Tr(\scrC_q^s | \widehat\Omega_f) \= \Tr(\scrV_q^s | \widehat{R[\Delta]}^+) - \Tr(\scrV_q^s | \scrD_{0,f}\bigl(\widehat{R[\Delta]}^+\bigr)) \= (1-q^s) \; \Tr(\scrV_q^s | \widehat{R[\Delta]}^+).
\ee
Here the second equality follows from the commutation relation $\scrV_q^s \circ \scrD_{0,f} = q^s\, \scrD_{0,f} \circ \scrV_q^s$. Traces on $\widehat{R[\Delta]}^+$ and $R\lb \Delta\rb^+$ are the same. 
It is clear from the definition of $\scrV_q$ that for every $s \ge 1$ one has $\scrV_q^s(\v X^\v 0)= \v X^\v 0 + $ terms with $u_0 \ge 1$, and hence $\Tr(\scrV_q^s |R \lb\Delta\rb^+) \= \Tr(\scrV_q^s | R \lb\Delta\rb) - 1$. Finally, we combine~\eqref{resolution-traces} with Proposition~\ref{trace-formula} and get
\[\bal
(q^s-1)^n \, & \Tr(\scrC_q^s | \widehat\Omega_f) \= - q^{-s}(q^s-1)^{n+1} \Bigl( \Tr(\scrV_q^s | R \lb\Delta\rb) - 1 \Bigr) \\
&\= - q^{-s} \Bigl( q^s \; \# Z_f(\F_{q^s}) \- (q^s-1)^n - (q^s-1)^{n+1} \Bigr)\\
&\= (q^s-1)^n - \# Z_f(\F_{q^s}) \= \# Z_{\T^n \setminus Z_f}(\F_{q^s}).
\eal\]
\end{proof}\appendix
\section{Point counting and an alternative construction of the Cartier operator}
Suppose that $R=\Z_q$ where $q=p^a$ and $\sigma$ is the standard $p$th power Frobenius lift 
satisfying $\sigma^a=id$. Then the $a$th iteration of the Cartier operator 
$\scrC_q:=\cartier^a$ maps $\hat\Omega_f$ to itself. It follows from the 
estimate~\eqref{cartier-entries-estimate} that modulo every power $p^s$ the image of
$\cartier$ has finite rank. By this reason the trace of $\scrC_q$ is a well defined $p$-adic value.
In this section we will prove
the following 

\begin{theorem}\label{pointcount}
The trace of $(q^s-1)^n \times \scrC_q^s$ on $\hat \Omega_f$ equals 
the number of points on $\T^n \setminus Z_f$ with coordinates $x_1,\ldots,x_n\in\F^\times_{q^s}$.
\end{theorem}

\begin{remark} Since $\scrC_q$ is divisible by $q$ on the submodule of formal derivatives $U_f \subset \widehat \Omega_f$, we conclude from Theorem~\ref{pointcount} that on the quotient $Q_f = \widehat\Omega_f / U_f$ one has
\[
\Tr(\scrC_q^s | Q_f) \equiv 1+(-1)^{n+1} \#Z_f(\F_{q^s}) \; \mod {q^s}. 
\]
The term $1$ on the right corresponds to the eigenvector $1 \in \Omega_f$, which has eigenvalue $1$.
If the Hasse--Witt matrix $\beta_p(\Delta)$ is invertible, then $Q_f$ is a $\Z_q$-vector space of finite dimension and it follows from the above congruence that the polynomial
\[
\det(1- T \scrC_q | Q_f) / (1-T) \in \Z_q[T]
\]
raised to the power $(-1)^{n}$
is the unit-root part of the zeta function of the hypersuface $Z_f \subset \T^n$ over $\F_q$. Note that in our standard basis in $Q_f$ (that is, images of $\omega_\v u \in \Omega_f, \v u \in \Delta_\Z$) the operator $\scrC_q$ is given by the transpose of 
\[
\Lambda \, \sigma(\Lambda) \, \ldots \, \sigma^{a-1}(\Lambda),
\]
where $\Lambda=\Lambda_\sigma$ is the matrix from Theorem~\ref{main1}.
\end{remark}

\bigskip

We will use a resolution of the module $\Omega_f$. This construction ties our crystals with the exponential modules in the literature, e.g. in~\cite{Ba93},
and exhibits a natural lift of our Cartier operator which possesses nice properties and hence might be
useful on its own. With the lifted Cartier operator, the point counting can be done using a version of
Dwork's trace formula, which is now a standard technique in $p$-adic analysis. 

From now on $R$ is a characteristic zero ring, we will impose more assumptions when needed. Let us introduce
the auxiliary variable $x_0$ and define the subring 
\[
R[\Delta] \subset R[x_0,x_1^{\pm 1},\ldots,x_n^{\pm 1}]
\]
as the span of monomials $\v X^\v u = x_0^{u_0} \ldots x_n^{u_n}$ with $\v u \in C(\Delta)_\Z$. 
We remind the reader that throughout the paper we denoted $\v x^{\v u} = x_1^{u_1} \ldots x_n^{u_n}$
for $\v u=(u_0,\ldots,u_n)$, so now we shall use the capital letter for $\v X^\v u = x_0^{u_0} \v x^{\v u}$.
We denote $F(\v X) = x_0 f(\v x) \in R[\Delta]$. The operations 
\[
\scrD_{i, f} \= x_i \frac{\partial}{\partial x_i} + x_i \frac{\partial F(\v X)}{\partial x_i} : R[\Delta] \to R[\Delta], \quad i = 0,\ldots, n
\]
are called \emph{twisted derivatives}. Formally, we have 
$\scrD_{i,f} = e^{-F} \circ \theta_i \circ e^F$, where $\theta_i = x_i \frac{\partial}{\partial x_i}$.
Twisted derivatives commute with each other.      

Let $R[\Delta]^+$ be the free $R$-module generated by $\v X^\v u$ with $\v u \in C(\Delta)_Z^+$.
It is an ideal in $R[\Delta]$ and twisted derivatives preserve $R[\Delta]^+$. The \emph{Laplace transform}
is the $R$-linear map $\scrR : R[\Delta]^+ \to \Omega_f$ given by
\[\v X^{\v u} \mapsto (-1)^{u_0} (u_0-1)! \frac{\v x ^\v u}{f(\v x)^{u_0}} \=  (-1)^{u_0}\omega_{\v u}.
\]
This Laplace transform was basically defined in~\cite[p.244]{Ka68} and~\cite[\S 7]{Ba93}. 
See also our Remark~\ref{whyfactorials}. It is clear that $\scrR$ is surjective. 

\begin{proposition}\label{resolving-omega-f} The kernel of the Laplace transform is given by
$\scrD_{0,f}(R[\Delta]^+)$. Under the induced isomorphism
\[
\scrR : R[\Delta]^+/ \scrD_{0,f}(R[\Delta]^+) \overset{\sim}\to \Omega_f 
\] 
the twisted defivative $\scrD_{i,f}$ corresponds to the usual derivative $\theta_i$ for each $1 \le i \le n$.
\end{proposition} 
\begin{proof} Under the Laplace transform the elements $\scrD_{0,f}
(\v X^\v u) = u_0 \v X^{\v u} + x_0 f(\v x) \v X^{\v u}$ are mapped to 
\[
(-1)^{u_0}u_0! \frac{\v x ^\v u}{f(\v x)^{u_0}} + (-1)^{u_0+1} u_0!  \frac{\v x ^\v u f(\v x)}{f(\v x)^{u_0+1}} \= 0.
\]  
It is clear that these elements generate all relations in $\Omega_f$ and therefore they span the kernel of $\scrR$. 

Let $1 \le i \le n$. Since twisted derivatives commute, $\scrD_{i,f}$ maps ${\rm Ker}(\scrR)={\rm Im}(\scrD_{0,f})$ to itself. The fact that the induced map on $\Omega_f$ coincides with $\theta_i$ can be easily checked on monomials:
\[\bal
\scrR \left(\scrD_{i,f} (\v X^\v u) \right) &\= \scrR \left( u_i \v X^\v u + \theta_i F(\v X) \v X^\v u \right)\= (-1)^{u_0} (u_0-1)! \left(  u_i \frac{\v x^{\v u}}{f(\v x)^{u_0}} - u_0 \frac{\v x^{\v u} \theta_i f(\v x)}{f(\v x)^{u_0+1}}\right)\\ & \= (-1)^{u_0} \theta_i \omega_\v u = \theta_i \scrR(\v X^\v u).
\eal\]
\end{proof} 
 
\begin{corollary}\label{dwork-deRham} $R[\Delta]^+/\sum_{i=0}^{n} \scrD_{i,f}(R[\Delta]^+) \cong \dwork$.
\end{corollary}

We would like to remark that in~\cite[Theorem 7.13]{Ba93} the quotient module on the left in this corollary was identified with $H^n_{dR}(\T^n \setminus Z_f)$ under the condition that $R$ is a field and $f(\v x)$ is $\Delta$-regular. At the end of the introductory section we mentioned the relation between Dwork modules and de Rham cohomology having in mind Corollary~\ref{dwork-deRham}.

To define the Cartier operator, we turn on our usual assumptions that $\cap_s p^s R = \{0\}$ and $R$ is $p$-adically complete. Let $R\lb \Delta \rb$ be the ring of formal power series with coefficients in $R$ and support in $C(\Delta)$. The $p$-adic completion $\widehat{R[\Delta]} = \underset{s} \lim R[\Delta]/p^s R[\Delta]$ consists of power series with infinitely growing $p$-adic valuation of coefficients:  
\[
\widehat{R[\Delta]} = \left\{ \sum_{\v u \in C(\Delta)_\Z} a_\v u \v X^\v u \;:\; a_\v u \in R, \quad \ord_p(a_\v u) \to \infty \hbox{ as } u_0 \to \infty \right\} \subset R\lb \Delta \rb. 
\]
We denote by $\widehat{R[\Delta]}^+$ the ideal of power series with zero constant term ($a_\v 0 = 0$). It follows from Proposition~\ref{resolving-omega-f} that $\hat\Omega_f \cong \widehat{R[\Delta]}^+/\scrD_{0,f}\Bigl(\widehat{R[\Delta]}^+\Bigr)$.
 
\begin{theorem}\label{cartier-coincide-thm} Consider the operator on power series given by 
\[
V_p \left( \sum_{\v u} a_{\v u} \v X^{\v u} \right) \=  \sum_{\v u} (-p)^{u_0} a_{p \v u} \, \v X^{\v u}.
\]
Let $p>2$. For every $p$th power Frobenius lift $\sigma: R \to R$, the operator 
\[
\scrV_{\sigma} = e^{-F^\sigma} \circ V_p \circ e^F
\]
maps $R\lb\Delta\rb$ to itself. Operator $\scrV_\sigma$ preserves $R\lb \Delta\rb^+$ and it is divisible by $p$ on this submodule. The following commutation relation with twisted derivatives
\be{p-commutation-cartier-lift}
\scrD_{i,f^\sigma} \circ \scrV_\sigma = p \, \scrV_\sigma \circ \scrD_{i,f}
\ee
holds for each $0 \le i \le n$. 

The operator $\scrV_\sigma$ preserves $\widehat{R[\Delta]}$ and the induced map $p^{-1} \scrV_\sigma: \widehat\Omega_f \to \widehat \Omega_{f^\sigma}$ coincides with the Cartier operator $\cartier: \widehat\Omega_f \to \widehat \Omega_{f^\sigma}$ constructed in Section~\ref{sec:cartier}.
\end{theorem}

\begin{proof} Fix any $\v u \in C(\Delta)_\Z^+$ and denote $\nu_0=p\lceil \frac{u_0}{p}\rceil - u_0$. Observe that
\[\bal
V_p \Bigl( e^{x_0 f(\v x) }\v X^\v u\Bigr) &\= \sum_{n \ge 0,\; p\,|(n+u_0)} \frac{(- p x_0)^{\frac{n+u_0}p} \cartier(f(\v x)^n \v x^\v u)}{n!} \\
&\= \sum_{m \ge 0} \frac{(- p x_0)^{m+\lceil \frac{u_0}{p}\rceil} \cartier(f(\v x)^{p m + \nu_0} \v x^\v u)}{(p m + \nu_0)!}
\eal\]
and therefore 
\[\bal
\scrV_\sigma \v X^{\v u} &\= \sum_{k \ge 0} \frac{(-x_0)^{k} f^\sigma(\v x)^k}{k!} \sum_{m \ge 0} \frac{(- p x_0)^{m+\lceil \frac{u_0}{p}\rceil} \cartier(f(\v x)^{p m + \nu_0} \v x^\v u)}{(p m + \nu_0)!}\\
&\=  \sum_{r \ge 0} (- p x_0)^{r + \lceil \frac{u_0}{p}\rceil} \, \cartier \Bigl( \v x^\v u \sum_{m+k=r} \frac{f(\v x)^{p m + \nu_0} f^\sigma(\v x^p)^k}{(pm+\nu_0)! \, k! p^k}\Bigr)\\
&\= \sum_{r \ge 0} (-p x_0)^{r + \lceil \frac{u_0}p\rceil} \cartier \Bigl( \v x^\v u \sum_{s = 0}^r  \frac{G(\v x)^{r-s} f(\v x)^{ps+\nu_0}}{(r-s)!} \gamma_{p s + \nu_0} \Bigr)\\ 
\eal\]
where we substituted $f^\sigma(\v x^p)=f(\v x)^p + p G(\v x)$ and recognised the sums
\[
\gamma_{p s + \nu_0} =  \sum_{m=0}^{s} \frac{1}{(pm+\nu_0)! \, (s-m)!p^{s-m}}\\
\]
as coefficients of the \emph{Dwork exponential} $e^{x_0+\frac{x_0^p}{p}} = \sum_{n \ge 0} \gamma_n x_0^n$. The following standard estimate of their $p$-adic order 
\be{dwork-exponential-estimate}
\ord_p(\gamma_n) \ge \bigl(\frac{p-1}{p^2}-\frac{1}{p-1}\bigr)n
\ee
implies that the matrix coefficients given by $\scrV_\sigma \v X^\v u = \sum_{\v v} \scrF_{\v u,\v v} \v X^\v v$ have $p$-adic valuations bounded by
\[\bal
\ord_p (\scrF_{\v u,\v v}) &\ge r + \lceil \frac{u_0}p\rceil + \underset{0 \le s \le r}\min \Bigl( (-\frac1{p-1})(r-s) + \bigl( \frac{p-1}{p^2} - \frac{1}{p-1}\bigr)\bigl(ps+\nu_0 \bigr) \Bigr) \\
&\= r + \lceil \frac{u_0}p\rceil - \frac{r}{p-1} +  \underset{0 \le s \le r}\min \Bigl(- \frac{s}p  \Bigr) + \bigl( \frac{p-1}{p^2} - \frac{1}{p-1}\bigr) \bigr(p\lceil \frac{u_0}p \rceil - u_0\bigl)  \\
&\= \bigl( \frac{p-1}{p} - \frac1{p-1}\bigr) \bigl( r + \lceil \frac{u_0}p\rceil  \bigr) - \bigl( \frac{p-1}{p^2} - \frac{1}{p-1} \bigr)u_0 \\
& \= \bigl( \frac{p-1}{p} - \frac1{p-1}\bigr) v_0 + \bigl(\frac{1}{p-1} - \frac{p-1}{p^2}\bigr)u_0,
\eal\] 
where we used $r = v_0 - \lceil \frac{u_0}{p}\rceil$. Since this valuation is non-negative, we conclude that all $\scrF_{\v u, \v v} \in R$ and hence $\scrV_\sigma$ maps the module of formal series $R\lb \Delta \rb$ to itself. We also observe that $\scrV_\sigma$ is divisible by $p$ on $R\lb \Delta \rb^+$. Moreover, $\scrV_\sigma$ preserves $\widehat {R[\Delta]}$ because when $v_0 \to \infty$ we have $\ord_p(\scrF_{\v u, \v v} \in R) \to \infty$ uniformly in $u_0$.

\bigskip   

The commutation relation~\eqref{p-commutation-cartier-lift} follows immediately from $\theta_i \circ V_p = p \, \theta_i \circ V_p$ and the fact that $\scrD_{i,f}=e^{-F}\circ \theta_i\circ e^F$. Since $\widehat \Omega_f = \widehat{R[\Delta]}^+ / \scrD_{0,f}( \widehat{R[\Delta]}^+)$, we have a well defined induced map $p^{-1} \scrV_\sigma: \widehat\Omega_f \to \widehat\Omega_f^\sigma$. Let us show that this induced map coincides with the Cartier operator $\cartier$ defined in Proposition~\ref{cartier-def}. Since $\scrR(\v X^\v w)=(-1)^{u_0} \omega_{\v w}$, we consider 
\be{difference-of-two-cartier}\bal
 & p^{-1}  \scrV_\sigma \v X^\v u \- \sum_{\v v \in C(\Delta)_\Z^+} (-1)^{v_0+u_0} \, F_{\v u, \v v} \v X^\v v \\
 & \= \frac1p \sum_{r \ge 0} (-p x_0)^{r + \lceil \frac{u_0}p\rceil} 
 \cartier \Bigl( \v x^\v u \sum_{s = 0}^r  G(\v x)^{r-s} f(\v x)^{ps+\nu_0} 
 \frac{\gamma_{ps+\nu_0} - \delta_{s,0} \Gamma_p(u_0)}{(r-s)!}    \Bigr),
\eal\ee
where we used formula~\eqref{cartiermatrix-entries} and substitution 
$\frac{(u_0-1)!}{(\lceil \frac{u_0}{p}\rceil - 1)!}=(-1)^{u_0} p^{\lceil \frac{u_0}{p}\rceil - 1} 
\Gamma_p(u_0)$. It is easy to check that the difference~\eqref{difference-of-two-cartier} equals
\be{difference-of-two-cartier-as-derivative}
\scrD_{0,f^\sigma} \Bigl( \frac1p  \sum_{r \ge 0} (-p x_0)^{r + \lceil \frac{u_0}p\rceil} \cartier \bigl( \v x^\v u \sum_{s = 0}^r G(\v x)^{r-s} f(\v x)^{ps+\nu_0} \frac{ \mu_s }{(r-s)!} \bigr)\Bigr),
\ee
where the coefficients $\mu_s \in \Q$ are determined by the recurrence
\[
\bigl(s+\lceil \frac{u_0}p\rceil\bigr) \mu_{s} \= p^{-1} \, \mu_{s-1} + \gamma_{p s + \nu_0} - \Gamma_p(u_0) \delta_{s,0}.
\]
(The initial term $\mu_0$ is also determined by this formula and convention $\mu_{-1}=0$.) We claim that
\be{mu-padic-rate}
\ord_p(\mu_s) \ge \bigl(-\frac{1}{p} - \frac{1}{p-1}\bigr) s + 1 - \lceil \frac{u_0}p\rceil , 
\ee
and hence inside of $\scrD_{0,f^\sigma}(\ldots)$ in~\eqref{difference-of-two-cartier-as-derivative} the $p$-adic valuation of coefficients of the polynomial next to $x_0^{r + \lceil \frac{u_0}p\rceil}$ can be estimated from below as
\[\bal
-1 + r + \lceil \frac{u_0}p\rceil + &\underset{0 \le s \le r}\min \ord_p\bigl(\frac{\mu_s}{(r-s)!}\bigr) \ge r + \underset{0 \le s \le r}\min \bigl(\bigl(-\frac{1}{p} - \frac{1}{p-1}\bigr) s -\frac1{p-1} (r-s)\bigr)\\
&\= r + \bigl(-\frac{1}{p} - \frac{1}{p-1}\bigr)r \= \bigl(\frac{p-1}{p} - \frac{1}{p-1}\bigr)r.   
\eal\] 
Since this valuation is non-negative and grows infinitely as $r \to \infty$, we conclude that~\eqref{difference-of-two-cartier-as-derivative} belongs to $\scrD_{0,f^\sigma}\bigl(\widehat{R[\Delta]^+}\bigr)$ and therefore $p^{-1} \scrV_\sigma = \cartier$ on $\widehat \Omega_f$.

It only remains to prove~\eqref{mu-padic-rate}. For this purpose we consider $f(\v x)=1$ in~\eqref{difference-of-two-cartier} and~\eqref{difference-of-two-cartier-as-derivative}. In this case $\scrV_\sigma$ is multiplication by the Dwork exponential $e^{x_0 + p^{-1} x_0^p}$, followed by $\cartier$ and the substitution $x_0 \mapsto -p x_0$. We shall denote $\scrV_\sigma$ simply by $\scrV$ and $\scrD_{0,f^\sigma} = x_0 \frac{d}{d x_0}+x_0$ by $\scrD$. The equality of~\eqref{difference-of-two-cartier-as-derivative} and~\eqref{difference-of-two-cartier} can be written as
\[\bal
\scrD \bigl( p^{-1} \sum_{r \ge 0} (-p x_0)^{r + \lceil \frac{u_0}p\rceil}  \mu_r \bigr) &\= p^{-1} \scrV ( x_0^{u_0}) \- p^{-1} \Gamma_p(u_0) (- p x_0)^{\lceil \frac{u_0}p\rceil} \\
&\= p^{-1} \scrV \bigl( x_0^{u_0} - \Gamma_p(u_0) x_0^{p \lceil \frac{u_0}p\rceil} e^{- x_0 - p^{-1} x_0^p}\bigr).
\eal\] 
Note that $\scrD$ is invertible on $\Q_p\lb x_0 \rb^+ = x_0 \Q_p\lb x_0 \rb$, and hence the commutation relation $\scrV \scrD = p \scrD \scrV$ can be rewritten as $p^{-1} \scrD^{-1} \scrV =  \scrV \scrD^{-1}$. Applying $\scrD^{-1}$ to the last identity and using the commutation relation, we get
\be{mu-series-via-cartier}
p^{-1} \sum_{r \ge 0} (-p x_0)^{r + \lceil \frac{u_0}p\rceil}  \mu_r  \= \scrV \scrD^{-1} \bigl( x_0^{u_0} - \Gamma_p(u_0) x_0^{p \lceil \frac{u_0}p\rceil} e^{- x_0 - p^{-1} x_0^p}\bigr).
\ee
Let 
\[
L_{\alpha,\beta} = \{ \sum_{m \ge 1} a_m x_0^m \in x_0 \Q_p\lb x_0 \rb \;|\; \ord_p(a_m) \ge \alpha m + \beta \}.
\]
Let $\alpha_0 = \frac{p-1}{p^2}-\frac1{p-1}$ and $\alpha_1=\frac{p-1}{p}-\frac1{p-1}$. Note that~\eqref{mu-padic-rate} precisely means that the series in~\eqref{mu-series-via-cartier} belongs to $L_{\alpha_1,0}$. In order to demonstrate this fact, we first notice that for any $\alpha \ge \alpha_0$ and any $\beta$ we have $\scrV: L_{\alpha,\beta} \to L_{\alpha_1,\beta}$. Indeed, since $e^{x_0+p^{-1}x_0^p} \in L_{\alpha_0,0}$ we decompose $\scrV$ into three steps and check that
\[
L_{\alpha,\beta} \overset{\cdot e^{x_0+p^{-1}x_0^p}}\to L_{\min(\alpha_0,\alpha),\beta} = L_{\alpha_0,\beta} \overset{\cartier}\to L_{p \alpha_0,\beta} \overset{x_0 \mapsto -px_0}\to L_{p \alpha_0+1,\beta}=L_{\alpha_1,\beta}.
\]
In the view of~\eqref{mu-series-via-cartier}, it now suffices to show that
\be{integral-todo-estimate}
\scrD^{-1} \bigl( x_0^{u_0} - \Gamma_p(u_0) x_0^{p \lceil \frac{u_0}p\rceil} e^{- x_0 - p^{-1} x_0^p}\bigr) \in L_{\alpha_0,0}.
\ee
It is useful to observe that $\scrD (x_0^m) = m x_0^m + x_0^{m+1}$ and 
\[
\scrD(x_0^{m} e^{- x_0 - p^{-1} x_0^p})= e^{-x_0} x_0 \frac{d}{d x_0} (x_0^{m} e^{- p^{-1} x_0^p}) = (m x_0^m - x_0^{m+p}) e^{- x_0 - p^{-1} x_0^p}.
\]
Using these two rules one can easily check that 
\be{int-monomial}
\scrD \bigl( (-1)^{u_0-1}(u_0-1)!\sum_{m=0}^{u_0-1} \frac{(-x_0)^m}{m!}  \bigr) =  x_0^{u_0} 
\ee
and 
\be{int-twisted-monomial}
\scrD \bigl( (n-1)! p^{n-1} \sum_{m=0}^{n-1} \frac{x_0^{p m}}{m!\, p^m}  e^{- x_0 - p^{-1} x_0^p}\bigr) =  - x_0^{p n} e^{- x_0 - p^{-1} x_0^p}.
\ee
Note that under $\scrD(\ldots)$ the polynomial in~\eqref{int-monomial} has integral coefficients and the series in~\eqref{int-twisted-monomial} belongs to $L_{\alpha_0,0}$ if one cuts off its constant term. We shall use~\eqref{int-twisted-monomial} with $n = \lceil \frac{u_0}p\rceil$.  Since $(-1)^{u_0}(u_0-1)!=\Gamma_p(u_0)(\lceil \frac{u_0}p\rceil-1)!p^{\lceil \frac{u_0}p\rceil-1}$, we get
\[\bal
\scrD^{-1} \bigl( x_0^{u_0} - &\Gamma_p(u_0) x_0^{p \lceil \frac{u_0}p\rceil} e^{- x_0 - p^{-1} x_0^p}\bigr)
\\
&\= (-1)^{u_0-1}(u_0-1)! \Bigl( \sum_{m=0}^{u_0-1} \frac{(-x_0)^m}{m!} - \sum_{m=0}^{\lceil \frac{u_0}{p}\rceil-1} \frac{x_0^{p m}}{m!\, p^m}  e^{- x_0 - p^{-1} x_0^p} \Bigr).
\eal\] 
Note that the constant term of the series in the right-hand side vanishes, which means that we integrated~\eqref{integral-todo-estimate} in $x_0\Q_p\lb x_0 \rb$ explicitly. Since $\Gamma_p(u_0)$ is a $p$-adic integer and $\alpha_0<0$, this series belongs to $L_{\alpha_0,0}$ due to the remarks made after~\eqref{int-monomial} and~\eqref{int-twisted-monomial}. This completes our proof of~\eqref{mu-padic-rate}.
\end{proof}

\begin{remark} One can easily define a connection on $R[\Delta]$ in a way and it commutes with the twisted derivatives. Namely, for every derivation $\delta: R \to R$ we define its action on $R[\Delta]$ as
\[
\nabla_\delta := \delta + \delta F,
\] 
where the first summand simply means that the derivation $\delta$ is applied to the coefficients and the second one means multiplication by the polynomial $\delta F (\v X) = x_0 (\delta f)(\v x)$. Formally, one can write $\nabla_\delta = e^{-F} \cdot \delta \cdot e^F$. To see that $\nabla_\delta$ commutes with the twisted derivatives, recall that $\scrD_{i,f}=e^{-F} \cdot \theta_i \cdot e^F$ and note that $\delta$ and $\theta_i$ commute. Operations $\nabla_\delta$ preserve $R[\Delta]^+$ and descend to its quotients by the images of twisted derivatives, particularly to $\Omega_f$ and $\dwork$. It is easy to check that $\nabla_\delta$ acts on $\Omega_f$ as the natural extension of $\delta$ to rational functions, the operation which we simply denoted by the same letter $\delta$ earlier in this paper. 

\bigskip

Finally, observe that the operator $\scrV_\sigma = e^{-F^\sigma} \cdot V_p \cdot e^F$ defined in Theorem~\ref{cartier-coincide-thm} commutes with the connection operators. Namely, it is obvious that $V_p$ commutes with $\delta$ as operators on power series, and after twisting by exponentials we obtain 
\[
\scrV_\sigma \cdot \nabla_\delta = \nabla^\sigma_\delta \cdot \scrV_\sigma,
\] 
where $\nabla^\sigma_\delta = \delta + \delta F^\sigma = e^{-F^\sigma} \cdot \delta \cdot e^{F^\sigma}$. This observation turns quotients of $\widehat{R[\Delta]}^+$ by twisted derivatives into crystals. 
\end{remark}

From now on we consider $R=\Z_q$ with $q=p^a$. Here we have the standard $p$th power Frobenius lift $\sigma: \Z_q \to \Z_q$ which satisfies $\sigma^a=id$. Consider the operator on power series $Z_q\lb \Delta \rb$  given by 
\[
\scrV_q := e^{-F} \circ V_p^a \circ e^F .
\]
Below we compute the traces of powers of $\scrV_q$ using a few standard tricks in $p$-adic analysis, which are basically due to Dwork. 

\begin{remark}\label{on-traces-definition} The traces are well-defined $p$-adic numbers because modulo every $p^s$ the operator $\scrV_q$ has finite-dimensional image (see the $p$-adic estimate of the matrix entries in the proof of Theorem~\ref{cartier-coincide-thm}).  Note also that the traces only depend on the mod $p$ reduction of the polynomial $F(\v X)=x_0 f(\v x)$. Indeed, if $F'(\v X)-F(\v X)=p G(\v X)$ then the respective operators on power series are conjugate $\scrV_q' = e^{- p G} \circ \scrV_q \circ e^{p G}$ and modulo each power of $p$ this identity can be written using matrices of finite size.
\end{remark}

\begin{proposition}\label{trace-formula} For all $s \ge 1$ one has
\[
(1-q^s)^{n+1} \Tr \left(\scrV_q^s \;|\; \Z_q \lb \Delta\rb \right) \= q^s \; \# Z_f(\F_{q^s}) \- (q^s-1)^n. 
\]

\end{proposition}
\begin{proof} Let $\pi \in \overline{\Q_p}$ be a number satisfying $\pi^{p-1}=-p$. We will work with Laurent series with coefficients in $R=\Q_q(\pi)$ and support in the cone $C(\Delta)$. Let $\rho: R\lb \Delta \rb \to R\lb \Delta \rb$ be the operation given by $\rho(\v X^\v u)=\pi^{u_0} \v X^\v u$. Note that $V_p = \rho^{-1} \circ \cartier \circ \rho$ and $e^{\pm F} = \rho^{-1} \circ e^{\pm \pi F} \circ \rho$, and hence
\be{conjugate-cartier}
\rho \circ \scrV_q^s \circ \rho^{-1} = e^{-\pi F} \circ \scrC_q^{s} \circ e^{\pi F} \= \scrC_q^{s} \circ \Theta_{F,s}
\ee
where we used the power series 
\be{theta-F}
\Theta_{F,s}(\v X) := \exp\bigl( \pi F(\v X) - \pi F(\v X^{q^s})\bigr) = \sum_{u \in C(\Delta)_\Z} b_{\v u} \v X^\v u.
\ee
From~\eqref{conjugate-cartier} it is clear that $\Tr(\scrV_q^s)=\sum_{\v u \in C(\Delta)_\Z} b_{(q^s-1) \v u}$. Since 
\[
\sum_{x \in \Z_q: x^{q-1}=1} x^u \= \bcs q-1, &\mbox{ if } (q-1)|u, \\ 0, &\mbox{ otherwise }, \ecs
\]  
this trace can be computed by summation of values of~\eqref{theta-F} over tuples of Teichm\"uller units in $\Z_{q^s}$:
\be{trace-formula-1}
\sum_{\v X \in \Z_{q^s}^{n+1} \; : \; x_i^{q^s-1} = 1} \Theta_{F,s}(\v X) \= (q^s-1)^{n+1} \sum_{\v u \in C(\Delta)_\Z} b_{(q^s-1) \v u}  \= (q^s-1)^{n+1} \Tr(\scrV_q^s).
\ee

To evaluate the sum on the left, consider the \emph{Dwork exponential} $\theta_p(z)=\exp(\pi z - \pi z^p)$. This series has $p$-adic radius of convergence $> 1$ and $\zeta_p := \theta_p(1) = 1 + \pi \mod {\pi^2}$ is a $p$th root of unity. For $k \ge 1$, let $\theta_{p^k}(z)=\exp(\pi z - \pi z^{p^k}) = \prod_{i=1}^{k-1}\theta_p(z^{p^i})$. The additive character $\psi_k: \F_{p^k} \to \Q_p(\pi)^\times$ given by $\psi_k(\bar x) = \zeta_p^{\Tr_{\F_{p^k}/\F_p}(\bar x)}$ is related to the Dwork exponential via $\psi_k(\bar x) = \theta_{p^k}(\Teich(\bar x))$.

Write $F(\v X) = x_0 f(\v x) = \sum a_\v u \v X^\v u$ and let $\bar F(\v X) = \sum \bar a_\v u \v X^\v u$ with $\bar a_\v u \in \F_q$ be the reduction of $F$ modulo $p$. Denote $a_{\v u}' = \Teich(\bar a_\v u)$ and $F'(\v X)=\sum a_\v u' \v X^\v u$. For any vector $\bar X =(\bar x_0,\ldots, \bar x_n) \in \F_{q^s}^{n+1}$ we have 
\[
\psi_{as}(\bar F(\bar{ \v X})) = \psi_{as}(\sum \bar a_\v u \bar {\v X}^\v u ) = \prod \theta_{q^s}(a_\v u' \Teich(\bar{\v X})^\v u) = \Theta_{F'}(\Teich(\bar{\v X})).
\]
Therefore the left-hand sum in~\eqref{trace-formula-1} for $F'$ can be evaluated as
\be{trace-formula-2}\bal
&\sum_{\v X \in \Z_{q^s}^{n+1} \; : \; x_i^{q^s-1} = 1} \Theta_{F',s}(\v X) \= \sum_{\bar{\v X} \in (F_{q^s}^\times)^{n+1}} \psi_{as}(\bar F(\bar{\v X})) \= \sum_{\bar {\v x} \in (F_{q^s}^\times)^{n}} \sum_{x_0 \in F_{q^s}}  \psi_{as}(\bar x_0 f(\bar{\v x}))\\
&\= \sum_{\bar {\v x} \in (F_{q^s}^\times)^{n}} \bcs q^s-1, &\mbox{ if } \bar f(\bar{\v x}) = 0 \\ -1 , &\mbox{ if  } \bar f(\bar{\v x}) \ne 0 \ecs \quad \= q^s \, \# Z_f(\F_{q^s}) \- (q^s-1)^n. 
\eal\ee
By Remark~\ref{on-traces-definition}, since $F'\is F \mod p$ traces of powers of $\scrV_q'=e^{-F'} \circ V_p^a \circ e^{F'}$ and $\scrV_q = e^{-F} \circ V_p^a \circ e^{F}$ are equal. Hence our claim follows from~\eqref{trace-formula-1} and~\eqref{trace-formula-2}.
\end{proof}

\begin{proof}[Proof of Theorem~\ref{pointcount}.] By Theorem~\ref{cartier-coincide-thm}, we have
\be{resolution-traces}
q^s \, \Tr(\scrC_q^s | \widehat\Omega_f) \= \Tr(\scrV_q^s | \widehat{R[\Delta]}^+) - \Tr(\scrV_q^s | \scrD_{0,f}\bigl(\widehat{R[\Delta]}^+\bigr)) \= (1-q^s) \; \Tr(\scrV_q^s | \widehat{R[\Delta]}^+).
\ee
Here the second equality follows from the commutation relation $\scrV_q^s \circ \scrD_{0,f} = q^s\, \scrD_{0,f} \circ \scrV_q^s$. Traces on $\widehat{R[\Delta]}^+$ and $R\lb \Delta\rb^+$ are the same. 
It is clear from the definition of $\scrV_q$ that for every $s \ge 1$ one has $\scrV_q^s(\v X^\v 0)= \v X^\v 0 + $ terms with $u_0 \ge 1$, and hence $\Tr(\scrV_q^s |R \lb\Delta\rb^+) \= \Tr(\scrV_q^s | R \lb\Delta\rb) - 1$. Finally, we combine~\eqref{resolution-traces} with Proposition~\ref{trace-formula} and get
\[\bal
(q^s-1)^n \, & \Tr(\scrC_q^s | \widehat\Omega_f) \= - q^{-s}(q^s-1)^{n+1} \Bigl( \Tr(\scrV_q^s | R \lb\Delta\rb) - 1 \Bigr) \\
&\= - q^{-s} \Bigl( q^s \; \# Z_f(\F_{q^s}) \- (q^s-1)^n - (q^s-1)^{n+1} \Bigr)\\
&\= (q^s-1)^n - \# Z_f(\F_{q^s}) \= \# Z_{\T^n \setminus Z_f}(\F_{q^s}).
\eal\]
\end{proof}

\end{document}